\documentclass[12pt, reqno]{amsart}
\usepackage[margin=1in]{geometry}
\usepackage[english]{babel}
\usepackage[latin1]{inputenc}
\usepackage{amssymb}
\usepackage{amsmath}
\usepackage{arydshln}
\usepackage{booktabs}
\usepackage{graphicx}
\usepackage{epsfig}
\usepackage{enumitem}
\usepackage{hyperref}
\usepackage{mathrsfs}
\usepackage{mathtools}
\usepackage[new]{old-arrows}
\usepackage[all]{xy}
\usepackage[usenames, dvipsnames]{color}
\usepackage{tikz}
\usepackage{graphicx}
\usepackage{caption}
\usepackage{subcaption}
\usepackage{hyperref}

\newcommand{\C}{\mathcal{C}}
\newcommand{\R}{\mathbb{R}}
\newcommand{\N}{\mathbb{N}}

\newcommand{\HH}{\mathcal{H}}
\newcommand{\W}{\mathcal{W}}
\newcommand{\A}{\mathcal{A}}
\newcommand{\B}{\mathcal{B}}
\newcommand{\K}{\mathcal{K}}
\newcommand{\V}{\mathcal{V}}
\newcommand{\U}{\mathcal{U}}
\newcommand{\Ss}{\mathcal{S}}

\newcommand{\OO}{{O}}
\newcommand{\SP}{${\,}^\prime$}

\newcommand{\Range}{\operatorname{Ran}}
\newcommand{\Ker}{\operatorname{Ker}}
\newcommand{\rank}{\operatorname{rank}}
\newcommand{\Span}{\operatorname{Span}}

\newcommand{\codim}{\operatorname{codim}}
\newcommand{\sgn}{\operatorname{sgn}}

\newcommand{\arxiv}[1]{\href{https://arxiv.org/abs/#1}{ArXiv:#1}}

\renewcommand{\a}{\boldsymbol{a}} 
\renewcommand{\b}{\boldsymbol{b}} 
\renewcommand{\c}{\boldsymbol{c}} 

\renewcommand{\A}{\boldsymbol{A}}
\renewcommand{\B}{\boldsymbol{B}}
\renewcommand{\C}{\boldsymbol{C}}

\newtheorem{theorem}{Theorem}[section]
\newtheorem{lemma}[theorem]{Lemma}

\newtheorem{cor}[theorem]{Corollary}

\newtheorem{prop}[theorem]{Proposition}

\theoremstyle{definition}
\newtheorem*{defi}{Definition}

\theoremstyle{remark}
\newtheorem{rem}[theorem]{Remark}
\newtheorem*{remark}{Remark}

\makeatletter
\let\c@table\c@figure 
\let\ftype@table\ftype@figure 
\makeatother

\title{Singular limits of sign-changing weighted eigenproblems}
\author{Derek Kielty}

\address{Department of Mathematics, University of Illinois, Urbana, IL 61801, U.S.A.}

\email{dkielty2@illinois.edu}


\begin{document}
\maketitle

\begin{abstract}
Consider the eigenvalue problem generated by a fixed differential operator with a sign-changing weight on the eigenvalue term. We prove that as part of the weight is rescaled towards negative infinity on some subregion, the spectrum converges to that of the original problem restricted to the complementary region. On the interface between the regions the limiting problem acquires Dirichlet-type boundary conditions. Our main theorem concerns eigenvalue problems for sign-changing bilinear forms on Hilbert spaces. We apply our results to a wide range of PDEs: second and fourth order equations with both Dirichlet and Neumann-type boundary conditions, and a problem where the eigenvalue appears in both the equation and the boundary condition. 
\end{abstract}


\section{\bf{Introduction}}
Heat conduction and vibrational modes of inhomogeneous materials are modeled by the weighted eigenvalue problem $- \Delta u = \lambda b u$, where $\lambda$ is an eigenvalue and $b$ is a positive weight function representing a pointwise heat capacity or mass density. When $b$ is allowed to change sign, new phenomena emerge modeling ecological population dynamics in the presence of a favorable ($b>0$), neutral ($b=0$), or unfavorable ($b<0$) food source (see the introduction of the monograph by Belgacem \cite{Bel}). Much of the intuition and many of the techniques used to study the spectrum and eigenfunctions are no longer applicable. In particular, these sign-changing problems typically have a discrete spectrum of eigenvalues that accumulate at both $+\infty$ and $-\infty$, in contrast with the positive weight case. 

In this paper we investigate the singular ``large negative weight limit" of such eigenvalue problems first in the Hilbert space setting and then applied to a variety of partial differential equations. Roughly speaking, we show that if $c(x)$ is a nonnegative weight then the eigenvalues of the problem with weight $b - tc$ converge as $t \to \infty$ to those of an eigenvalue problem on the subregion $\Omega = \{c = 0\}$ with weight $b$ and mixed boundary conditions. For example, the positive eigenvalues of the Neumann problem
\begin{equation}
\label{Neumann}
\begin{cases} 
      -\Delta u = \lambda (b - tc) u & \text{ on a domain } \OO \\
      \hspace{2.5mm} \partial_n u = 0 & \text{ on } \partial \OO
   \end{cases}
\end{equation}
converge, as $t \to +\infty$, to the positive eigenvalues of the mixed boundary value problem
\begin{equation}
\label{eq:NeumannMixed}
\begin{cases} 
      -\Delta u = \lambda b u & \text{ on the subregion } \Omega = \{c = 0\} \\
      \hspace{2.5mm} \partial_n u = 0 & \text{ on } \partial \Omega \setminus \Gamma\\
      \hspace{7mm} u = 0 & \text{ on } \Gamma,
   \end{cases}
\end{equation}
where $\Gamma$ is the hypersurface that forms the interface between the set $\{c>0\}$ and its complement $\{c = 0\}$ (see Figure \ref{fig:Domain}). The key to establishing the convergence is an ``$L^2$-draining" inequality, as explained in Remark \ref{rem:CDrain}. This implies the eigenfunctions converge weakly to zero in $L^2(\Omega', c \, dx)$ as $t \to \infty$, and therefore, the limiting eigenfunction is supported in $\Omega$ (see Figure \ref{fig:NeumannEF}).

We are primarily concerned with singular limits of sign-changing eigenproblems. We formulate and prove our results in a general Hilbert space framework, as developed by Auchmuty \cite{Auch}; see also \cite{AuchRiv}. (A formulation in terms of compact operators may also be possible.) In Section \ref{sec:PDEApp} we apply our singular limit convergence theorems to various PDE eigenvalue problems (summarized in Table \ref{tab:ProbTab}). Among others, our results cover some fourth order equations involving the Bi-Laplacian and problems with a variety of boundary conditions, and a problem where the eigenvalue appears both in the equation and the boundary condition. 

In particular, our results apply to problems that are not coercive. For example, the left side of the above Neumann problem (\ref{Neumann}) is generated by the $L^2$-norm of the gradient, which is not coercive on $H^1$. In positively weighted problems, there are various ways around this lack of coercivity, but many of these techniques fail or are more complicated when the weight changes sign. 

In Section \ref{sec:MatPen} we apply our results to a weighted version of the traditional matrix eigenvalue problem, known as a \emph{matrix pencil}: $A v = \lambda B v$. In this setting we are able to give a complete description of the behavior of the spectrum as $t \to \infty$. Sections \ref{sec:Lemmas} and beyond are devoted to proofs of the main results and applications.

\begin{figure}

\setlength{\unitlength}{0.8cm}
\begin{picture}(9,4)
\thicklines
\qbezier(3.5,2.5)(5,1)(3,.05)
\put(.7,2.3){$\OO$}
\put(2,1.2){$\Omega$}
\put(5.8,1.2){$\Omega'$}
\put(4.3,.4){$\Gamma$}

\begin{tikzpicture}
\draw (2,2) ellipse (3cm and 1cm);
\end{tikzpicture}

\end{picture}

\caption{The domain $\OO$, with subregions $\Omega = \{c = 0\}$ and $\Omega' = \{c>0\}$, and interface $\Gamma$. See Section \ref{AppPrec} for precise definitions. }
\label{fig:Domain}
\end{figure}
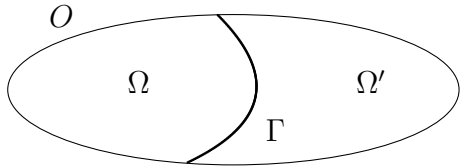

\subsection*{Motivation}

Positive eigenfunctions of the weighted Laplacian can be interpreted as a population densities, because the eigenproblem is the linearization of the steady-state of a nonlinear model for population dynamics (see the introduction of \cite{Bel}). From this perspective, our limit of eigenproblems can be interpreted as the limit of ecological models in which the food source (the weight) becomes arbitrarily unfavorable (negative) on some subregion. In ecology, Dirichlet boundary conditions are known as ``hostile" boundary conditions. Our results make rigorous the following heuristic: a region with arbitrarily unfavorable food source creates a hostile boundary at its interface with the complementary region. This heuristic is analogous to how a Dirichlet boundary condition for a Schr\"{o}dinger eigenproblem can arise from deeper and deeper potential wells converging to an infinite potential well.

\subsection*{Literature}
Recently, a special case of the aformentioned convergence as $t \to \infty$ for the Neumann problem (\ref{Neumann}) was established by Mazzoleni, Pellacci, and Verzini to study optimal design problems \cite[Lemma 1.2]{MPV1}, \cite{MPV2}. This convergence allowed them to transfer information about optimizers from the mixed Dirichlet--Neumann problem (\ref{eq:NeumannMixed}) to the Neumann problem (\ref{Neumann}) for large $t$. In addition to this optimal design result, there has been work on extremizing the first positive and negative eigenvalues over weights with constraints on the extreme and average values. This problem was investigated by Cox \cite{Cox} for the Dirichlet Laplacian with positive weights, for the Neumann Laplacian by Lou and Yanagida \cite{LY}, for a nonlinear Neumann Laplacian problem by Derlet, Gossez, and Tak\'a\v c \cite{DGT}, and for the Robin Laplacian by Lamboley et al.\ \cite{LLNP}. The resulting extremizers are often of \emph{bang-bang} type, meaning their range consists only of the extreme values.  

Recently, other phenomena have been studied for problems with sign-changing weights. The analog of the Weyl asymptotic holds for the eigenvalues of the Laplace-Beltrami operator with a sign-changing weight was established by Bandara, Nursultanov, and Rowlett in \cite{BNR}. Their results hold for \emph{rough Riemannian manifolds}, that is, Riemannian manifolds with metrics that are only assumed to be bounded and measurable. In a nonlinear setting, Kaufmann, Rossi, and Terra \cite{pLap} studied limits of the $p$-Laplacian eigenvalue problem with a sign-changing weight as $p$ tends to infinity. They showed that the asymptotics of the positive eigenvalues are controlled by the geometry of the set where the weight is positive, which generalized results from the unweighted setting.

\begin{figure}
    \centering
    \includegraphics[scale =.5]{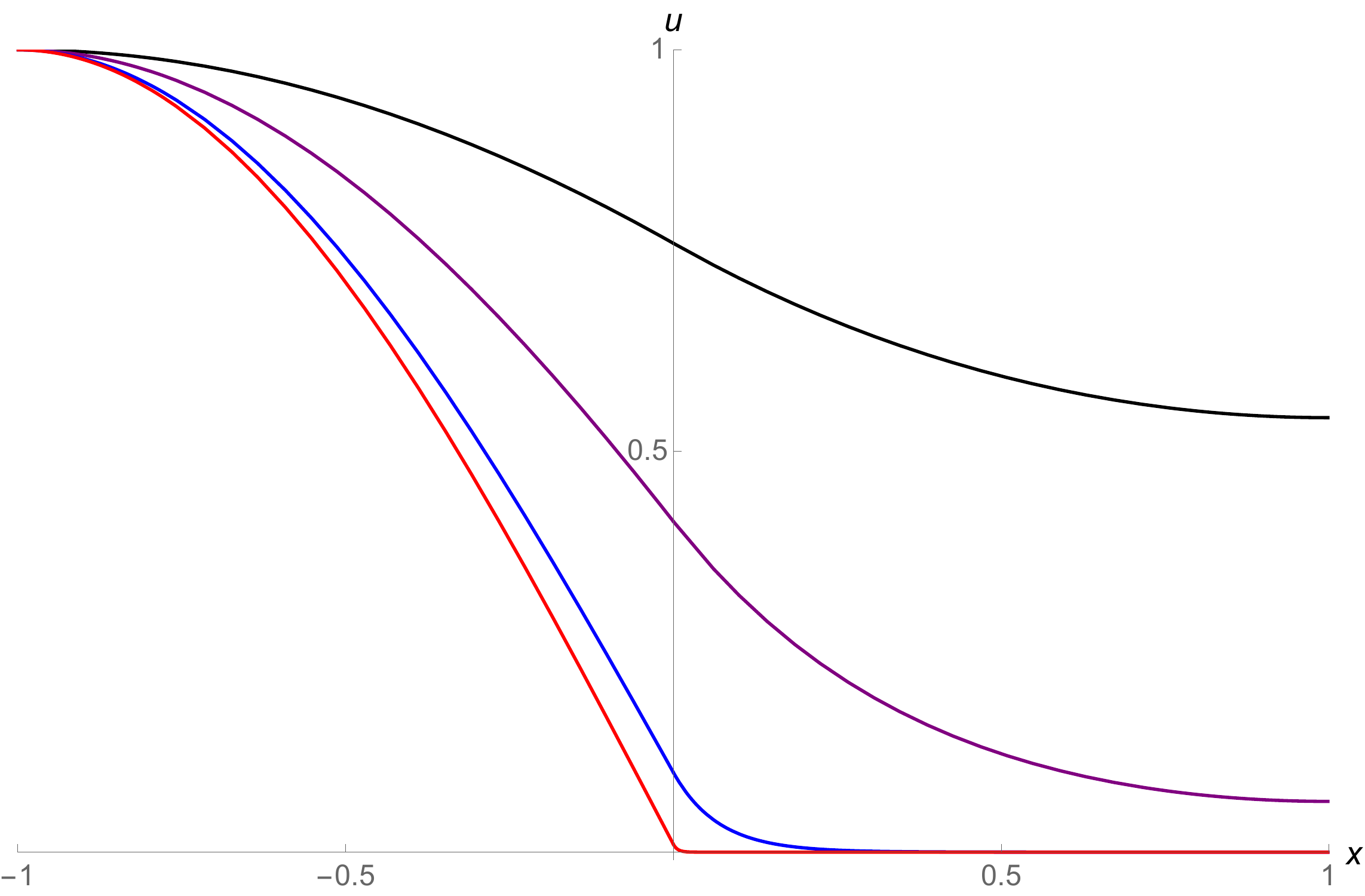}
    \caption{The $L^{\infty}$-normalized eigenfunction corresponding to the first positive eigenvalue of the Neumann problem (\ref{Neumann}) with domain $\Omega = (-1,1)$ and weight $b^t = \chi_{[-1,0]} - t\chi_{[0,1]}$, plotted for $t = 1.5,5,100,10^5$. The values $u(0)$ are decreasing with increasing $t$-values. The eigenfunctions have the form $A \sin(\sqrt{\lambda}x) + B \cos(\sqrt{\lambda}x)$ on $(-1,0]$ and $C\sinh(\sqrt{\lambda t}x) + D\cosh(\sqrt{\lambda t}x)$ on $[0,1)$. Observe that $u$ approaches zero on $\Omega' = (0,1)$ so that the eigenfunctions acquire Dirichlet boundary conditions on $\Gamma = \{0\}$ in the limit.}
    \label{fig:NeumannEF}
\end{figure}

\section{\bf{Main Results}}
\label{sec:MR}

\subsection*{Conditions for Convergence of Spectrum}

The set-up for our main results consists of a Hilbert space $(\HH, \langle \cdot,\cdot \rangle)$ with norm $\lVert \cdot \rVert$, three symmetric bilinear forms $\a,\b,\c : \HH \times \HH \to \R$, and the family of bilinear forms given by
\[\b^t = \b - t \c, \quad \text{for each} ~ t \in \R.\]
The associated quadratic forms are  
\[\A(u) = \a(u,u), \quad \B(u) = \b(u,u), \quad \C(u) = \c(u,u), \quad \text{and} \quad \B^t(u) = \b^t(u,u).\]

In what follows we allow $\B$ to be a sign-changing function, but impose that 
\[\mathbf{C}(\cdot) ~\text{is nonnegative, and not identically zero on} ~ \HH.\] 
We seek solutions $(\lambda^t,u^t) \in \R \times (\HH \setminus \{0\})$ to the eigenequation  
\begin{equation}
\label{eq:GenEigEq}
\a(u^t,v) = \lambda^t \b^t(u^t,v), \quad \text{for all}~ v \in \HH.
\end{equation}
Call $\lambda^t$ the eigenvalue and $u^t$ the eigenvector.

We denote this problem by the triple $(\HH,\a,\b^t)$ and will define other eigenvalue problems using the same ``\emph{space-form-form}" triple notation. For example, the weak formulation of the weighted Dirichlet or Neumann Laplacian eigenvalue problem is generated by taking $\HH = H^1_0(\OO)$ or $H^1(\OO)$, $\a(u,v) = \int_{\OO} \nabla u \cdot \nabla v \, dx$, and $\b^t(u,v) = \int_{\OO} uvb^t \, dx$ with $b^t = b - tc$ as in (\ref{Neumann}). In what follows we define the kernel of a bilinear form $\b$ to be
\[\Ker(\b) = \{u \in \HH : \b(u,v) = 0 ~ \text{for all} ~ v \in \HH\}.\]

To identify the appropriate limiting problem as $t \to \infty$ let 
\begin{equation*}
    \K = \Ker(\mathbf{c}).
\end{equation*}
Observe that $\b^t$ is ``stationary" on $\K$ in the sense that $\b^t = \b$ on $\K \times \HH$ for all $t$. We show that, in a certain sense, the eigenvalues of the problem $(\HH,\a,\b^t)$ converge to those of the limiting problem $(\K,\a,\b)$ as $t \rightarrow \infty$.

Two sets of conditions will give rise to two versions of this limiting statement. The Dirichlet and Neumann Laplacian eigenvalue problems are model problems for these sets of conditions, respectively. The first set of conditions is: 

\bigskip

\textbf{(C1):} $\A(\cdot)$ is a coercive on $\HH$, meaning $\gamma \lVert \cdot \rVert^2 \leq \A(\cdot)$ for some $\gamma > 0$.

\textbf{(C2):} $\a(\cdot,\cdot)$ is continuous on $\HH \times \HH$.

\textbf{(C3):} $\b(\cdot, \cdot)$ and $\c(\cdot,\cdot)$ are weakly  (sequentially) continuous on $\HH \times \HH$.

\bigskip

\noindent In condition (C3), a bilinear form $\b(\cdot,\cdot)$ is weakly sequentially continuous if $\b(u^n,v^n) \to \b(u,v)$ whenever $u^n \rightharpoonup u$ and $v^n \rightharpoonup v$, where ``$\rightharpoonup$" denotes weak convergence in $\HH$. In what follows we will say ``weakly continuous" in place of ``weakly sequentially continuous" for brevity.

\begin{rem}
\label{rem:WeakCont}
Note that $\b(\cdot,\cdot)$ and $\c(\cdot,\cdot)$ being \emph{weakly continuous} on $\HH \times \HH$ is, in fact, a stronger condition than them simply being continuous. This follows from the fact that if $u^n$ and $v^n$ are norm convergent in the Hilbert space $\HH$ with limits $u$ and $v$, then $u^n \rightharpoonup u$ and $v^n \rightharpoonup v$, so that $\b(u^n,v^n) \rightarrow \b(u,v)$.
\end{rem}

Condition (C1) is designed to handle problems that are coercive on all of $\HH$. To handle problems that fail to be coercive on a finite dimensional subspace (such as the Neumann Laplacian, whose associated bilinear form annihilates the constants), we now develop a variant of (C1). The condition is:

\bigskip

\textbf{(C1$'$):} $\A(\cdot)$ is coercive on $\Ker(\a)^{\perp}$, with $\dim(\Ker(\a)) < \infty$, and $\Ker(\a) \cap \K = \{0\}$,

\bigskip
\noindent where $\Ker(\a) = \{w \in \HH : \a(w,v) = 0 ~ \text{for all} ~ v \in \HH \}$ and $\perp$ denotes the orthogonal complement with respect to the $\HH$-inner product.

The second set of conditions is then (C1$'$) together with (C2) and (C3) from above, and in this case, we restrict the problem $(\HH,\a,\b^t)$ to the ``moving" Hilbert space
\[\HH^t = \{u \in \HH : \b^t(u,w) = 0 \text{ for all } w \in \Ker(\a)\}.\]
The spaces $\HH^t$ should be thought of as the $\b^t$-orthogonal complement of $\Ker(\a)$ for each $t \in \R$. In this ``moving" setting we prove that the eigenvalues of $(\HH^t,\a,\b^t)$, and therefore the nonzero eigenvalues of $(\HH,\a,\b^t)$, converge to those of the limiting problem $(\K,\a,\b)$.

In our PDE applications, $\Ker(\a)$ is some finite dimensional subspace of the polynomials, and the coercivity condition in (C1$'$) is  established by a generalization of the Poincar\'e inequality for mean zero functions. For example, in the Neumann Laplacian case, $\Ker(\a)$ consists of the constants, and the ``moving" Hilbert space $\HH^t$ consists of functions $u$ satisfying $\int_{\OO} u b^t \, dx = 0$.

The following terminology will distinguish the above two cases:

\begin{defi}
We call conditions (C1)--(C3) the \textbf{fixed Hilbert space conditions} and  conditions (C1$'$),(C2),(C3) the \textbf{moving Hilbert space conditions}.
\end{defi}

\subsection*{Existence of Spectrum}
When (C1) and (C2) hold, the bilinear form $\a(\cdot,\cdot)$ induces an inner product on $\HH$ that is equivalent to the $\HH$-inner product. Let $\oplus_{\a}$ denote the $\a$-orthogonal direct sum. The following theorem is a consequence of existence results due to Auchmuty \cite{Auch}:

\begin{theorem}[Existence of spectrum]
\label{thm:ExistSpec}
Assume that $\b(\cdot,\cdot) \not\equiv 0$. If $(\HH,\a,\b)$ satisfies (C1),(C2), and (C3) then there exists a (possibly finite) sequence of nonzero eigenvalues
\[ \dots \leq \lambda_{-j} \leq \dots \leq \lambda_{-2} \leq \lambda_{-1} < 0 < \lambda_1 \leq \lambda_2 \leq \dots \leq \lambda_j \leq \cdots,\]
which have finite multiplicity and accumulate only at $\pm \infty$. Moreover, we have the decomposition
\[\HH = \U_- \oplus_{\a} \U_+ \oplus_{\a} \U_{\infty},\]
where $\U_{\pm}$ is the norm closed span of the eigenvectors with positive $(+)$ or negative $(-)$ eigenvalues and $\U_{\infty} = \Ker(\b)$.
\end{theorem}

\subsection*{Application to \texorpdfstring{$\b^t$}{Lg}}

When $\b^t(\cdot,\cdot) \not\equiv 0$ for all $t$, applying Theorem \ref{thm:ExistSpec} with $\b^t$ instead of $\b$ immediately gives existence of spectrum for $(\HH,\a,\b^t)$ in the fixed Hilbert space case and with $\HH^t$ and $\b^t$ instead of $\HH$ and $\b$ gives existence of spectrum for $(\HH^t,\a,\b^t)$ in the moving Hilbert space case (since $\A$ is coercive on $\HH^t$ by Lemma \ref{MCoercivity} when $|t|$ is large). In the latter case, Lemma \ref{lem:MEE} shows that $(\HH,\a,\b^t)$ has the same spectrum as $(\HH^t,\a,\b^t)$ up to zero eigenvalues when $|t|$ is large. See point 1 of the discussion at the end of this section for further explanation of the hypothesis on $t$ in the moving case. 

Note that $\b^t(\cdot,\cdot) \equiv 0$ for at most one $t \in \R$, say $t_0$, since $\c(\cdot,\cdot) \not\equiv 0$. In this case, $\b^t(\cdot,\cdot) = (t_0 - t)\c(\cdot,\cdot)$ and each eigenvalue of $(\HH,\a,\b^t)$ is simply a $(t_0-t)^{-1}$ times an eigenvalue of $(\HH,\a,\c)$ when $t \neq t_0$.

In what follows let $J_+,J_- \in \{0,1,2, \dots, \infty\}$ denote, respectively, the number of positive and negative eigenvalues of the limiting problem $(\K,\a,\b)$. For each $t \in \R$, write
\[\lambda_{\pm j}^t = \lambda_{\pm j}(\HH,\a,\b^t), \quad \text{for each} ~ j \in \N,\] 
for the $j^{\text{th}}$ positive ($+$) or negative ($-$) eigenvalue of $(\HH,\a,\b^t)$ counting multiplicity. These eigenvalues satisfy the eigenequation (\ref{eq:GenEigEq}) for some corresponding eigenvectors $u^t_{\pm j}$.

\subsection*{Convergence of Spectrum}

The first of our main results accounts for all the positive eigenvalues as $t \to \infty$ of $(\HH,\a,\b^t)$ via the dichotomy: if $\lambda_j(\K,\a,\b)$ exists then $\lambda_j^t$ converges to it; otherwise $\lambda_j^t$ tends to $+\infty$. In what follows, we will write $\lambda^t \nearrow \lambda$ as $t \to \infty$ to mean: $\lambda^t \to \lambda$ as $t \to \infty$, and either $\lambda^t$ is increasing for all $t \in \R$ (in the fixed Hilbert space setting) or increasing for all $t$ sufficiently large (in the moving Hilbert space case). We will use ``$\nearrow$" similarly when $t$ increases to a finite value. The notation ``$\searrow$" is also defined analogously. How large $t$ must be only depends on $\HH,\a,\b,$ and $\c$ (see Lemma \ref{PseudoMon} and Remark \ref{rem:T}).

\begin{theorem}[Convergence and blow-up of positive spectrum as $t \to +\infty$] \label{SpecConv} Assume either the fixed or moving Hilbert space conditions hold. \\
(i) If $j \leq J_+$ then $\lambda_j^t \nearrow \lambda_j(\K,\a,\b)$ as $t \nearrow \infty$. \\
(ii) If $j > J_+$ and $\lambda_j^{t_*}$ exists for some $t_* \in \R$ ($t_*$ sufficiently large in the moving Hilbert space case) then $\lambda_j^t \nearrow +\infty$ as $t \nearrow t_j$ for some $t_j \in (t_*, +\infty]$.
\end{theorem}
\noindent Observe that when $J_+ = \infty$, $j \leq J_+$ means $j < \infty$ and part (ii) is vacuous. For matrix pencils, the blow-up of eigenvalues in part (ii) of the theorem and related phenomenon can be seen in Figure \ref{fig:MatPen} and are proven in Theorem \ref{prop:MatPen}.

In the moving Hilbert space setting define 
\[\HH^{\infty} = \liminf_{t \to \infty} \HH^t = \cup_{s} \cap_{t > s} \HH^t,\]
and note that $\HH^{\infty}$ is a closed subspace of $\HH$ by Lemma \ref{subspace}, where we also find an explicit characterization of $\HH^{\infty}$. In the next proposition, let $\codim(\K)$ denote the codimension of $\K$ as a subspace of $\HH$, and $\codim_{\HH^{\infty}}(\K \cap \HH^{\infty})$ denote the codimension of $\K \cap \HH^{\infty}$ as a subspace of $\HH^{\infty}$.

\begin{prop}[Convergence of negative spectrum as $t \to +\infty$]
\label{NegConv}
If the fixed Hilbert space conditions hold then $\lambda_{-j}(\HH,\a,\b^t)$ exists for all sufficiently large $t$ and increases to zero as $t \rightarrow \infty$, for each $j =1,2,\dots, \codim(\K)$. If the moving Hilbert space conditions hold then $\lambda_{-j}(\HH,\a,\b^t)$ exists for all sufficiently large $t$ and tends to zero as $t \rightarrow \infty$, for each $j =1,2, \dots, \codim_{\HH^{\infty}}(\K \cap \HH^{\infty})$.
\end{prop}
\noindent Recall that a Riesz representation argument shows that $\c(u,v) = \langle u,C v \rangle$ for a bounded symmetric operator $C : \HH \to \HH$ since $\c(\cdot,\cdot)$ is norm continuous by Remark \ref{rem:WeakCont}. Combining this with the above proposition shows that $\rank(C)$-many negative eigenvalues of $(\HH,\a,\b^t)$ increase to zero in the fixed Hilbert space case.

Applying the above two results as $t \searrow -\infty$, we have the following convergence statements: 

\begin{cor}[Convergence of spectrum as $t \searrow -\infty$]
\label{cor:NegInfLim}
Assume either the fixed or moving Hilbert space conditions hold with $t$ sufficiently negative in the moving case.
\begin{enumerate}
\item[(i)] If $j \leq J_-$ then $\lambda_{-j}^t \searrow \lambda_{-j}(\K,\a,\b)$ as $t \searrow -\infty$.
\item[(ii)] If $j > J_-$ and $\lambda_{-j}^{t_*}$ exists for some $t_* \in \R$ ($t_*$ sufficiently negative in the moving Hilbert space case) then $\lambda_{-j}^t \searrow -\infty$ as $t \searrow t_{-j}$ for some number $t_{-j} \in [-\infty, t_*)$.
\item[(iii)] $\lambda_j^t$ exists for sufficiently negative $t$ and tends to zero as $t \searrow -\infty$, for $j = 1, \dots, \codim(\K)$ in the fixed Hilbert space case and for $j = 1, \dots, \codim_{\HH^{\infty}}(\K \cap \HH^{\infty})$ in the moving Hilbert space case.
\end{enumerate}

\end{cor}

\subsection*{Discussion}
1. Although our results concern the whole spectrum of indefinite problems, the behavior of low eigenvalues guide our approach to problems that fail to be coercive such as the Neumann problem (\ref{Neumann}). In the moving Hilbert space case, when $|t|$ is small it is possible for a negative eigenvalue to increase through zero and become positive as $t$ increases. This causes the $j^{\text{th}}$ eigenvalue to have a jump discontinuity and decrease, before it increases again. This phenomena illustrates why we restrict to $t$ sufficiently large to prove the $j^{\text{th}}$ eigenvalue is monotone in $t$ and for $\A$ to be coercive on $\HH^t$.

In particular, the principal eigenvalue (the one with a positive eigenfunction) of the Neumann problem (\ref{Neumann}) has a sign that is the opposite of the sign of $\int_{\OO} b^t \, dx$ (see \cite[Corollary 2.2.8]{Bel}). For an eigenvalue problem coming from a parabolic equation with dynamical boundary conditions (see Table \ref{tab:ProbTab}), Bandle, von Below, and Reichel \cite[Theorem 21]{BvBR} proved that there is a smooth curve of principal eigenvalues that passes through zero as $t$ (a parameter in the boundary condition) is varied.

2.  Convergence Theorem \ref{SpecConv} shows that each positive eigenvalue of the limiting problem is obtained as a limit of approximating eigenvalues, but for the negative eigenvalues, Proposition \ref{NegConv} says only that a certain number of them tend to zero. It does not assert that other negative eigenvalues tend to the negative spectrum of the limiting problem. In finite dimensions, negative eigenvalues do in fact converge to the negative spectrum of the limiting problem, by Proposition \ref{prop:MatPen} below, but in infinite dimensions the situation can be more complicated.

For example, Proposition \ref{NegConv} implies when $\codim(K)$ is infinite that $\lambda_{-j}^t$ tends to zero for every $j \geq 1$, making it difficult to imagine in what sense the negative spectrum of the approximating problem could be said to converge to the negative spectrum of the limiting problem. The problem is seen particularly clearly for a 1-dimensional Sturm--Liouville problem $-u'' = \lambda b^t u$ with Dirichlet boundary conditions when $b$ and $c$ are continuous. In this case, the spectrum consists of simple eigenvalues for each $t$ even when $b^t$ changes sign (see \cite[\S 10.72]{Ince} or \cite[Theorem B]{MGL}) and Proposition \ref{ContSpec} implies that $(t,\lambda_{-j}^t)$ is a continuous curve of eigenvalues in the $t \lambda$-plane. Therefore, each curve of negative eigenvalues cannot cross and must tend to zero. In what sense (if any) could these eigenvalues be said to approach the negative eigenvalues of the limiting Sturm--Liouville problem? Further work is needed to understand this situation.

In general, spectral curves can cross, making it possible for a curve of negative eigenvalues to converge to a negative limiting eigenvalue. In order for this to happen, the indices of the eigenvalues forming such a curve must get larger and larger as the curve is crossed by successively many eigenvalue curves tending to zero. Examples with this behavior can be constructed using diagonal operators on $\ell^2(\mathbb{N})$.

\section{\bf{Applications to Partial Differential Equations}}
\label{sec:PDEApp}

Now we apply our Hilbert space convergence theorems from the previous section to prove that the spectrum of each approximating problem in Table \ref{tab:ProbTab} converges to the spectrum of its corresponding limiting problem, as described by Proposition \ref{prop:PDEApp}. Convergence of the spectrum for the problems in the first and second halves of Table \ref{tab:ProbTab} is proved via the fixed and moving Hilbert space versions of Theorem \ref{SpecConv}, respectively. It follows from Lemma \ref{lem:Strong} that eigenfunctions of the approximating problems converge in Sobolev norm to corresponding eigenfunctions of the limiting problem. 

\begin{table}
\small
\renewcommand{\arraystretch}{1.5}
    \begin{tabular}{cll}
    \toprule
        & \multicolumn{1}{c}{Approximating Problem} & \multicolumn{1}{c}{Limiting Problem} \\
        
        \midrule 
        \addlinespace[6pt]
        
        \shortstack{Schr\"odinger operator ($V \geq 0$) \\ Dirichlet Laplacian ($V \equiv 0$)} &
        
    $\begin{aligned}
       \begin{cases} 
      (-\Delta + V) u^t = \lambda^t b^t u^t & \text{ on } \OO \\
     \hspace{17mm} u^t = 0 & \text{ on } \partial \OO 
        \end{cases} 
    \end{aligned}$ & 
        
    $\begin{aligned}
       \begin{cases} 
      (-\Delta + V) u = \lambda b u \hspace{-2mm}& \text{ on } \Omega \\
      \hspace{17mm} u = 0 \hspace{-2mm}& \text{ on } \partial \Omega
       \end{cases}
    \end{aligned}$ \\
    
    \addlinespace[12pt]    
    \shortstack{Robin Laplacian \\ $(\alpha > 0)$} &
        
    $\begin{aligned}
        \begin{cases}
        \hspace*{\fill} -\Delta u^t = \lambda^t b^t u^t & \text{ on } \OO \\
        \partial_n u^t + \alpha u^t = 0 & \text{ on } \partial \OO
          \end{cases}
    \end{aligned}$ &
      
    $\begin{aligned}
        \begin{cases}
        \hspace{\fill} -\Delta u = \lambda b u & \text{ on } \Omega \\
        \hspace{\fill}  \partial_n u + \alpha u = 0 & \text{ on } \Gamma^c\\
        \hspace{\fill} u = 0 & \text{ on } \Gamma
          \end{cases}
    \end{aligned}$ \\
   
    \addlinespace[12pt]   
      
    \shortstack{Clamped Bi-Laplacian\\ ($\tau \geq 0$)} &
        
    $\begin{aligned}    
        \begin{cases} 
       (\Delta \Delta - \tau \Delta) u^t = \lambda^t b^t u^t & \text{ on } \OO \\
        \hspace{7mm} u^t = \partial_n u^t = 0 & \text{ on } \partial \OO
        \end{cases}
   \end{aligned}$ &
   
    $\begin{aligned}  
        \begin{cases} 
        (\Delta \Delta - \tau \Delta) u = \lambda b u & \text{ on } \Omega \\
        \hspace{8mm} u = \partial_n u = 0 & \text{ on } \partial \Omega 
        \end{cases}
   \end{aligned}$ \\
    
    \addlinespace[6pt]  
    \hdashline
    \addlinespace[6pt]

    \shortstack{Neumann Laplacian \\ \, \\ \, \\ \,} & 
        
    $\begin{aligned}
       \begin{cases} 
      -\Delta u^t = \lambda^t b^t u^t & \text{ on } \OO \\  \hspace{2.5mm} \partial_n u^t = 0 & \text{ on } \partial \OO 
        \end{cases} 
    \end{aligned}$ & 
        
    $\begin{aligned}
       \begin{cases} 
      -\Delta u = \lambda b u & \text{ on } \Omega \\
      \hspace{2.5mm} \partial_n u = 0 & \text{ on } \Gamma^c\\
      \hspace{6.5mm} u = 0 & \text{ on } \Gamma
       \end{cases}
    \end{aligned}$ \\
    
    \addlinespace[12pt]
    
    \shortstack{Free Bi-Laplacian\\ ($\tau \geq 0$)} &
        
    $\begin{aligned}    
        \begin{cases} 
        (\Delta \Delta - \tau \Delta) u^t = \lambda^t b^t u^t \hspace{-2mm}& \text{ on } \OO \\
        \hspace{3.5mm} Mu^t = Nu^t = 0 \hspace{-2mm}& \text{ on } \partial \OO
        \end{cases}
   \end{aligned}$ &
   
    $\begin{aligned}  
        \begin{cases} 
        (\Delta \Delta - \tau \Delta) u = \lambda b u \hspace{-2mm}& \text{ on } \Omega \\
        \hspace{4.5mm} Mu = Nu = 0 \hspace{-2mm}& \text{ on } \Gamma^c\\
        \hspace{8.5mm} u = \partial_nu = 0 \hspace{-2mm}& \text{ on } \Gamma
        \end{cases}
   \end{aligned}$ \\
   
   \addlinespace[12pt]
    
    \shortstack{Laplacian with Dynamical \\Boundary Conditions} &
        
        $\begin{aligned}
       \begin{cases} 
    \hspace{1mm}  -\Delta u^t = \lambda^t u^t & \text{ on } \OO \\
      -\partial_n u^t =  t  \lambda^t u^t & \text{ on } \partial \OO 
        \end{cases} 
    \end{aligned}$ & 
        
    $\begin{aligned}
       \begin{cases} 
      -\Delta u = \lambda u & \text{ on } \Omega = \OO \\
    \hspace{7mm} u = 0 & \text{ on } \partial \Omega = \partial \OO
       \end{cases}
    \end{aligned}$ \\
    
    \addlinespace[6pt]
    \bottomrule
    \addlinespace[6pt]
    
    \end{tabular}
    
    \caption{Under suitable assumptions on the domains and weights, as explained in Section \ref{AppPrec}, the eigenvalues of (the weak formulations of) each approximating problem converge (with multiplicity) to those of the corresponding limiting problems by Theorem \ref{SpecConv}. In the limiting problems, $\Gamma^c = \partial \Omega \setminus \Gamma$. In the Schr\"odinger operator example, we take $V \in L^{\infty}(\OO)$ to be nonnegative. In the Free Bi-Laplacian example, the boundary operators are $Mu = \partial_n^2 u$ and $Nu = \tau \partial_n u - \text{div}_{\partial \OO}(P_{\partial \OO}[(D^2u)n]) - \partial_n (\Delta u)$, where $P_{\partial \OO}$ is the operator that projects a vector at $y \in \partial \OO$ onto the tangent space of $\partial \OO$ at $y$.}
    \label{tab:ProbTab}
    \vspace{-5mm}
\end{table}

\subsection{Standing assumptions and definitions}
\label{AppPrec}

First we identify the relevant spaces $\HH$ and $\K$ and the bilinear forms $\a,\b,\c$ in order to formulate each approximating and limiting problem (except for the Laplacian with dynamical boundary conditions). Let $\OO \subset \R^d$ be a bounded Lipschitz domain, $k \geq 1$ an integer, and $b,c \in L^{p(d,k)}(\OO)$ where the exponent is
\begin{equation}
\label{eq:pdk}
p(d,k) = \begin{cases}
      d/2k & \text{if} \quad d \geq 2k+1,\\
      1 & \text{if} \quad d = 1,\dots, 2k.
\end{cases}
\end{equation}
The $\a$ forms for each problem will be given in Table \ref{tab:SpaceForm}. Using the functions $b$ and $c$ we define the bilinear forms
\[\b(u,v) = \int_{\OO} uv b \, dx \quad \text{and} \quad \c(u,v) = \int_{\OO} uv c \, dx.\]
Similarly define
\[b^t = b - tc \quad \text{and} \quad \b^t(u,v) = \int_{\OO} uv b^t \, dx = \b(u,v) - t\c(u,v).\]

\begin{table}
\small
          \renewcommand{\arraystretch}{1.5}
    \begin{tabular}{cccc}
    \toprule
         & Space $\HH$ & Form $\a(u,v)$ & Space $\K$ \\
        \midrule 
        \addlinespace[6pt]
        \shortstack{Schr\"odinger operator ($V \geq 0$), \\ Dirichlet Laplacian ($V \equiv 0$)} & 
        
    $\begin{aligned}
    H^1_0(\OO) 
    \end{aligned}$ & 
        
    $\begin{aligned}
    \int_{\OO} \nabla u \cdot \nabla v + uvV \, dx
    \end{aligned}$ &
    
    $\begin{aligned}
    H^1_0(\Omega)
    \end{aligned}$ \\
    
    \addlinespace[12pt]    
    
    \shortstack{Robin Laplacian\\ ($\alpha > 0$)} &
        
    $\begin{aligned}
    H^1(\OO)
    \end{aligned}$ &
      
    $\begin{aligned}
    \int_{\OO} \nabla u \cdot \nabla v \, dx + \alpha \int_{\partial \OO} uv \, dS    \end{aligned}$ &
    
    $\begin{aligned}
    H^1_{\Gamma}(\Omega)
    \end{aligned}$\\
    
    \addlinespace[12pt]
      
    \shortstack{Clamped Bi-Laplacian\\ ($\tau \geq 0$)} &
        
    $\begin{aligned}    
    H^2_0(\OO)
   \end{aligned}$ &
   
    $\begin{aligned}  
    \int_{\OO} (\Delta u) (\Delta v) + \tau \nabla u \cdot \nabla v \, dx
   \end{aligned}$ &
   
   $\begin{aligned}  
   H^2_0(\Omega)
   \end{aligned}$ \\
    
    \addlinespace[6pt]  
    \hdashline
    \addlinespace[6pt]   
    
    Neumann Laplacian & 
        
    $\begin{aligned}
    H^1(\OO) 
    \end{aligned}$ & 
        
    $\begin{aligned}
    \int_{\OO} \nabla u \cdot \nabla v \, dx
    \end{aligned}$ &
    
    $\begin{aligned}
    H^1_{\Gamma}(\Omega)
    \end{aligned}$ \\
    
    \addlinespace[12pt]
      
    \shortstack{Free Bi-Laplacian\\ ($\tau \geq 0$)} &
        
    $\begin{aligned}    
    H^2(\OO)
   \end{aligned}$ &
   
    $\begin{aligned}  
    \int_{\OO} D^2u \cdot D^2v + \tau \nabla u \cdot \nabla v \, dx
   \end{aligned}$ &
   
   $\begin{aligned}  
   H^2_{\Gamma}(\Omega)
   \end{aligned}$ \\
   
   \addlinespace[12pt]    
    
    \shortstack{Laplacian with Dynamical \\Boundary Conditions} &
        
    $\begin{aligned}
    H^1(\OO)
    \end{aligned}$ &
      
    $\begin{aligned}
    \int_{\OO} \nabla u \cdot \nabla v \, dx
    \end{aligned}$ &
    
    $\begin{aligned}
    H^1_0(\OO)
    \end{aligned}$\\
    
    \addlinespace[6pt]
    \bottomrule
    \addlinespace[6pt]
    
    \end{tabular}
    
    \caption{For each $\HH$ and $\a(\cdot,\cdot)$ listed above, the eigenvalues of $(\HH,\a,\b^t)$ converge to those of $(\K,\a,\b)$ as described by Proposition \ref{prop:PDEApp}. Here $\b^t(u,v) = \int_{\OO} uvb^t \, dx$ except in the Dynamical Boundary Conditions case where $\b^t(u,v) = \int_{\OO} uv \, dx - t \int_{\partial \OO} u v \, dS$. In the latter case $\Omega = \OO$. In the Free Bi-Laplacian case, $D^2u$ is the $d^2$-dimensional Hessian vector consisting of all second derivatives of $u$.}
    \label{tab:SpaceForm}
    \vspace{-5mm}
\end{table}

Assume $c(x) \geq 0$ for all $x$, and that \[\Omega' = \{c > 0\} \quad \text{and} \quad \Omega = \OO \setminus \overline{\Omega'},\] are nonempty open sets with Lipschitz boundary (as defined in \cite{EG}). 
In Remark \ref{rem:Omega} we discuss the possibility of relaxing the above hypotheses on $\OO,\Omega,$ and $\Omega'$.

In what follows $H^k(\Omega)$ is the usual Sobolev space $W^{k,2}(\Omega)$. Let 
\[\Gamma = \partial \Omega \cap \partial \Omega',\] 
and let the space
\[H^k_{\Gamma}(\Omega) = \{u \in H^k(\Omega) : T(D^{\beta}u) = 0 ~\text{on}~ \Gamma ~ \text{for each multiindex} ~ \beta ~ \text{with} ~ |\beta| \leq k-1\},\]
where $T : H^1(\Omega) \to L^2(\partial \Omega)$ is the trace operator. Observe that $H^k_{\Gamma}(\Omega)$ is a closed subspace of $H^{k}(\Omega)$ since $T \circ D^{\beta}$ is continuous on $H^k(\Omega)$. 

We show (in Lemma \ref{GenApp}) that if $\HH = H^k(\OO)$ or $H^k_{\partial \OO}(\OO)$ then the kernel of $\c$ is 
\[\K = \{u \in \HH : u \equiv 0 ~\text{on}~ \Omega'\},\] consisting of the functions that vanish on $\Omega'$. While $\K$ is the correct limiting space given by the convergence Theorem \ref{SpecConv}, it is more natural to consider $\widetilde \K$, which we define as the space of functions in $\K$ restricted to $\Omega$. 

To identify the space $\widetilde{\K}$, in Lemma \ref{GenApp} we show when $\HH = H^k(\OO)$ or $H^k_{\partial \OO}(\OO)$ that $\widetilde{\K} = H^k_{\Gamma}(\Omega)$ or  $H^k_{\partial \Omega}(\Omega)$, respectively. Since $\a(\cdot,\cdot)$ is defined by integration in our applications, and functions in $\K$ vanish on $\Omega'$, we can restrict the integration to $\Omega$ to obtain a new bilinear form $\widetilde \a$ on $\widetilde \K$ and similarly for $\b$ and $\c$. For the remainder of this section, we identify $\K$ with $\widetilde \K$ and $\a,\b,$ and $\c$ with $\widetilde \a, \widetilde \b,$ and $\widetilde \c$, respectively.

For the Laplacian with dynamical boundary conditions $\OO = \Omega$ and
\[\b^t(u,v) = \b(u,v) - t\c(u,v) = \int_{\OO} uv \, dx - t \int_{\partial \OO} uv \, dS.\]
In this case we will show that $\K = H^1_0(\OO)$ in Lemma \ref{DBCLem}.

\begin{remark}
It is known that $H^k_{\partial \Omega}(\Omega)$ can be characterized as the closure of $C_0^{\infty}(\Omega)$ when $\partial \Omega$ is sufficiently regular (see \cite[\S 2.4.4]{Necas}). In particular, this characterization holds for $k \in \{1,2\}$ when $\partial \Omega$ is Lipschitz so we will write $H^k_0(\Omega)$ for $H^k_{\partial \Omega}(\Omega)$ and similarly for the spaces on $\OO$ in these cases. While it seems plausible that $H^k_{\Gamma}(\Omega)$ could be constructed as the closure of $C^{\infty}_0(\Omega \sqcup \Gamma^c)$, where $\Gamma^c = \partial \Omega \setminus \Gamma$, we will work solely with the above definition of $H^k_{\Gamma}(\Omega)$.
\end{remark}

\subsection{PDE Convergence Results}

Now we construct triples $(\HH,\a,\b^t)$ and $(\K,\a,\b)$ as in Table \ref{tab:SpaceForm} by making choices of the Hilbert space $\HH$ and a bilinear form $\a(\cdot,\cdot)$. These triples correspond to weak formulations of the approximating and limiting eigenvalue problems for the partial differential operators considered in Table \ref{tab:ProbTab}. Let $\lambda_{\pm j} = \lambda_{\pm j}(\K,\a,\b)$ for each $j \geq 1$. Applying our results from Section \ref{sec:MR} to each of these triples, we obtain:

\begin{prop}
\label{prop:PDEApp}
Consider as above the domains $\OO,\Omega,\Omega'$, the weights $b,c$ and their associated bilinear forms $\b,\c$. For each problem in the first or second half of Table \ref{tab:SpaceForm}, the fixed or moving Hilbert space conditions hold, respectively, and $\K$ is the space indicated in the Table. For each $j \geq 1$:\\
(i) If $\{b|_{\Omega} > 0\}$ has positive measure then $\lambda_j^t$ and $\lambda_j$ both exist and $\lambda_j^t \nearrow \lambda_j$ as $t \to \infty$.\\
(ii) If $b|_{\Omega} \leq 0$ a.e. then $\lambda_j$ does not exist, and if $\lambda_j^{t_*}$ exists for some $t_* \in \R$ ($t_*$ sufficiently large for the problems in the second half of Table \ref{tab:SpaceForm}) then $\lambda_j^t \nearrow +\infty$ as $t \nearrow t_j$ for some $t_j \in (t_*,+\infty]$.\\
(iii) If $t$ is sufficiently large then $\lambda_{-j}^t$ exists and $\lambda_{-j}^t \to 0$ as $t \to \infty$, except when $d = 1$ in the Dynamical Boundary Conditions problem, in which case there is a single negative eigenvalue that tends to zero as $t \to +\infty$.
\end{prop}

The proposition is proved in Section \ref{sec:ProofPDEApp}. Proposition \ref{prop:PDEApp} could easily be strengthened to hold for problems with more general symmetric elliptic operators, but we choose to restrict the applications to the Laplacian and Bi-Laplacian for simplicity.


\section{\textbf{Application to Matrix Pencils and Blow-up Phenomenon}}
\label{sec:MatPen}

When $\HH$ is finite dimensional the eigenvalue problem $(\HH,\a,\b^t)$ is a variant of the traditional eigenvalue problem from linear algebra. In this setting we are able to obtain a complete description of the spectrum as $t \to \infty$. This example also illustrates that the range of indices for which Proposition \ref{NegConv} holds is as large as possible in the fixed Hilbert space case.

Let $A,B,$ and $C$ be symmetric $d \times d$ matrices. Assume that $A$ is positive definite, and that $C$ is positive semi-definite, nonzero, and has a nontrivial kernel. Let $B^t = B - t C$ and consider the \emph{matrix pencil} eigenvalue problem
\begin{equation*}
    A v^t = \lambda^t B^t v^t \quad \text{for} \quad (\lambda^t,v^t) \in \R \times (\R^d \setminus \{0\}), \quad t \in \R.
\end{equation*}
Denote this eigenvalue problem by the triple $(\R^d,\a,\b^t)$, where $\a(u,v) = u \cdot (A v)$, and $\b,\c,$ and $\b^t = \b - t \c$ are defined similarly. The fixed Hilbert space conditions hold and $\Ker(\c) = \Ker(C)$, so that $(\Ker(\c),\a,\b)$ is the limiting problem.

\begin{figure}
    \centering
    \includegraphics[scale = .5]{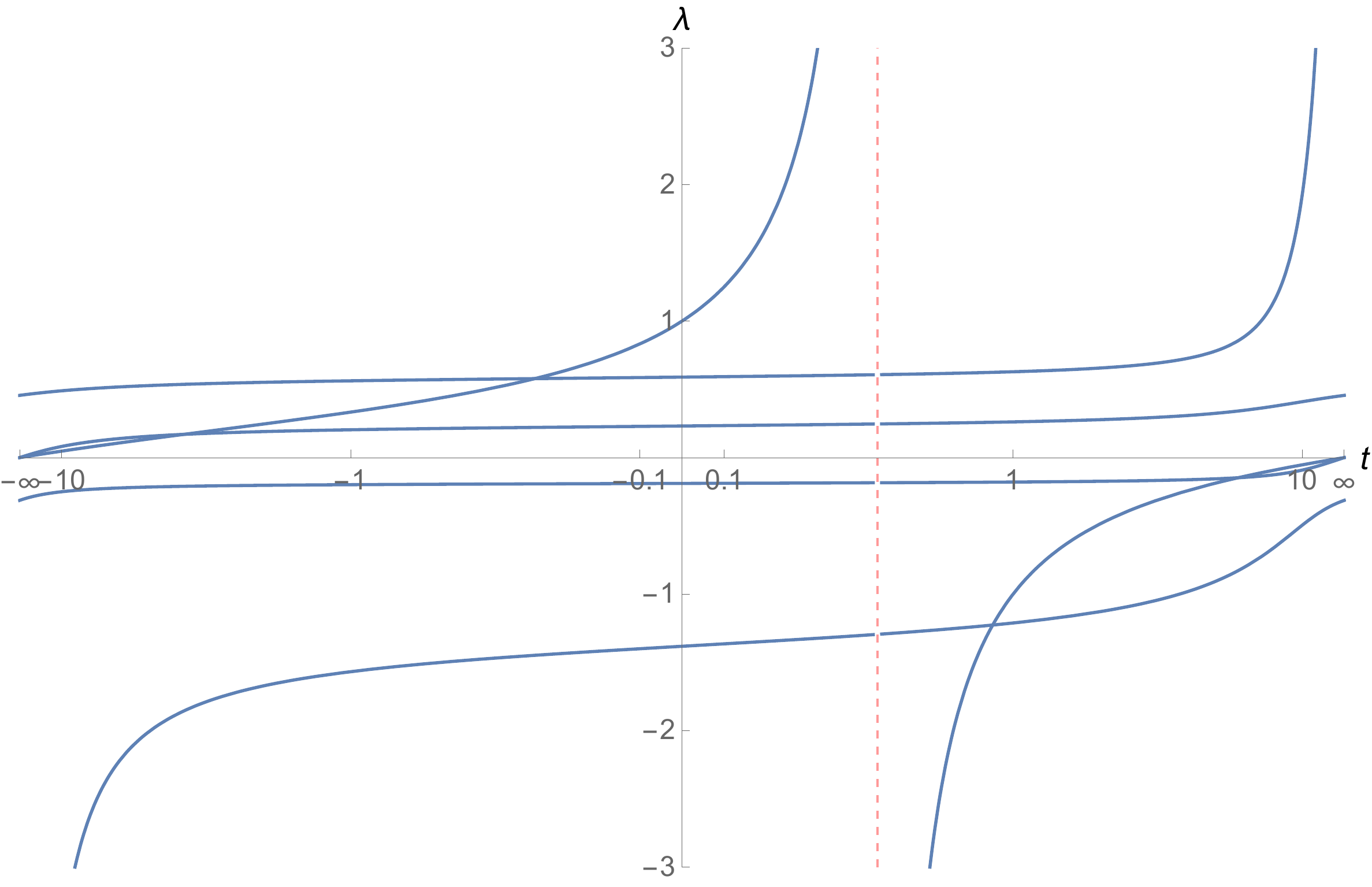}
    \caption{Eigenvalues of a matrix pencil problem $(\R^5,\a,\b^t)$ are plotted using a nonlinear scale such that the right and left endpoints of the horizontal axis represent $t = \pm \infty$. $A = I_{5 \times 5}$ is the $5 \times 5$ identity matrix and $B$ and $C$ are as in (\ref{eq:MatBC}). Apart from two negative eigenvalues that increase to zero, the eigenvalues of $(\R^d,\a,\b^t)$ converge to those of $(\Ker(\c),\a,\b)$ as $t \to +\infty$, including a positive eigenvalue that blows-up to the eigenvalue-at-infinity of the limiting problem. Notice the positive eigenvalue that blows-up in finite-time ``reappears" as a negative eigenvalue, near the dashed vertical asymptote.}
    \label{fig:MatPen}
\end{figure}

Due to the form of the eigenequation it is natural to say that $(\R^d,\a,\b^t)$ has an \emph{eigenvalue-at-$\infty$} of multiplicity $\dim(\Ker(\b^t))$ when $\Ker(\b^t)$ is nontrivial. We view the eigenvalues-at-$\infty$ as genuine eigenvalues in this section. Recall that $J_{\pm}$ are the number of positive $(+)$ and negative $(-)$ eigenvalues of $(\Ker(\c),\a,\b)$. Similarly, let $J_{\infty}$ denote the number of eigenvalues-at-$\infty$ of the limiting problem $(\Ker(\c),\a,\b)$, so that $J_{\infty} = \dim(\Ker(\b|_{\Ker(\c)}))$. Let $\lambda_{\pm j}$ denote the eigenvalues of the limiting problem $(\Ker(\c),\a,\b)$. The next proposition will be proved at the end of Section \ref{sec:ProofPDEApp}.

\begin{prop}[Matrix Pencil Convergence]
\label{prop:MatPen}
If the matrices $A, B,$ and $C$ are as above and $\Ker(\b) \cap \Ker(\c)$ is trivial then:
\begin{enumerate}
    \item[(i)] $\lim_{t \to \infty} \lambda_j^t = + \infty$ for $J_+ + 1 \leq j \leq J_+ + J_{\infty}$;
    \item[(ii)] $\lim_{t \to \infty} \lambda_j^t = \lambda_j$ for $1 \leq j \leq J_+$;
    \item[(iii)] $\lim_{t \to \infty} \lambda_{-j}^t = 0$ for $1 \leq j \leq \rank(C)$;
    \item[(iv)] $\lim_{t \to \infty} \lambda_{-j}^t = \lambda_{-j+r}$ for each $\rank(C) + 1 \leq j \leq \rank(C) + J_-$.
\end{enumerate}
\end{prop}

In the proposition we require that the $\Ker(B) \cap \Ker(C)$ is trivial only to simplify the statement. The proposition can be modified to account for $\Ker(B) \cap \Ker(C)$ being nontrivial by observing that the problem $((\Ker(B) \cap \Ker(C))^{\perp_{\a}},\a,\b^t)$ has the same spectrum as $(\R^d,\a,\b^t)$ after adding an eigenvalue-at-$\infty$ of multiplicity $\dim(\Ker(B) \cap \Ker(C))$ for each $t \in \R$.

Positive eigenvalues that tend to $+\infty$ in finite time ``reappear" from $-\infty$ as negative eigenvalues. This and other phenomena can be seen in Figure \ref{fig:MatPen}, where the eigenvalues of $(\R^5,\a,\b^t)$ are plotted with $A = I_{5 \times 5}$ and

\begin{equation}
\label{eq:MatBC}
B = \left(
\begin{array}{rrrrr}
 0 & 0 & 0 & -2 & 0 \\
 0 & 2 & -1 & 2 & 0 \\
 0 & -1 & -3 & 3 & 0 \\
 -2 & 2 & 3 & 1 & 0 \\
 0 & 0 & 0 & 0 & \,1 \\
\end{array}
\right) \quad \text{and} \quad C = \left(\begin{array}{ccccc}
 0 & \,0 & \,0 & \,0 & \,0 \\
 0 & 0 & 0 & 0 & 0 \\
 0 & 0 & 0 & 0 & 0 \\
 0 & 0 & 0 & 1 & 0 \\
 0 & 0 & 0 & 0 & 2 \\
\end{array}
\right).
\end{equation}
Note that since $A$ is the identity, the eigenvalues plotted in the figure are just the reciprocals of the eigenvalues of the matrix $B^t$.


\section{\bf{Preliminary Lemmas}}

\label{sec:Lemmas}


Recall that $\K = \Ker(\c) = \{u \in \HH : \c(u,v) = 0 ~ \text{for all} ~ v \in \HH\}$ and $\HH^{\infty} = \liminf_{t \to \infty} \HH^t$, where $\HH^t = \{u \in \HH : \b^t(u,w) = 0 \text{ for all } w \in \Ker(\a)\}$.

\begin{lemma}[Subspace Lemma]
\label{subspace}
If (C3) holds, then $\K,\HH^t,$ and $\HH^{\infty}$ are closed subspaces of $(\HH, \langle \cdot, \cdot \rangle)$. Consequently, $(\K, \langle \cdot, \cdot \rangle)$ and $(\HH^t, \langle \cdot, \cdot \rangle)$ are Hilbert spaces.  Whether or not (C3) holds we have
\begin{equation}
\label{eq:Hinf}
\HH^{\infty} = \{u \in \HH : \b(u,w) = \c(u,w) = 0 ~ \text{for all} ~ w \in \Ker(\a)\}.
\end{equation}
\end{lemma}

\begin{proof} In either the fixed or moving Hilbert space setting, the map $c_v : \HH \to \R$ defined by $c_v(u) = \c(u,v)$ is a norm continuous linear functional on $\HH$ for each $v \in \HH$ by condition (C3). Observe that $\K = \cap_{v \in \HH} \Ker(c_v)$ by the definition of $\K$. Thus, $\K$ is a closed subspace since it is the intersection of closed subspaces. In the moving Hilbert space setting, the same argument holds for $\HH^t$, and once we prove (\ref{eq:Hinf}) it will hold for $\HH^{\infty}$ as well.

By definition of $\HH^t$, the right side of (\ref{eq:Hinf}) is contained in $\HH^t$ for all $t$. Thus, it is contained in $\HH^{\infty}$. Let $u \in \HH^{\infty}$ so that $\b(u,w) = t\c(u,w)$ for all $w \in \Ker(\a)$ and for all $t$ sufficiently large. Since the right side depends on $t$ but the left side does not, we must have $\c(u,w) = 0$ for all $w \in \Ker(\a)$. Consequently, $\b(u,w) = 0$ for all $w \in \Ker(\a)$ as well. Thus, $u$ is an element of the right side of (\ref{eq:Hinf}) and the equality holds.

\end{proof}


\subsection*{Moving Hilbert space preliminary lemmas}


To prove our convergence theorems in the moving Hilbert space case we show that $\A(\cdot)$ is uniformly coercive on $\HH^t$ (Lemma \ref{MCoercivity}) and establish that $\lambda_j(\HH,\a,\b^t)$ is increasing for large $t$ and bounded from above (Lemma \ref{LimExists}). To show that the eigenvalues are increasing in $t$, it is not enough that the function $t \mapsto \B^t(u)$ is decreasing for each $u \in \HH$ since the subspace $\HH^t$ depends on $t$. 

The proof of the following lemma generalizes a calculation by Bandle and Wagner \cite[\S 2]{BW} for the first eigenvalue of the dynamical boundary conditions problem to the Hilbert space setting. To do this define
\[T = \frac{M_{|\B|}}{m_{\C}} + 1,\]
where $M_{|\B|}$ and $m_{\C}$ are, respectively, the maximum and minimum of $|\B|$ and $\C$ over the unit sphere of $\Ker(\a)$. In the proof of the lemma we show that $-\sgn(t) \b^t(\cdot,\cdot)$ is an inner-product on $\Ker(\a)$ for $|t| > T$. In this case let $\{w_1, \dots, w_m\}$ be a $(-\sgn(t) \b^t)$-orthonormal basis for $\Ker(\a)$ and $Q^t : \HH \to \HH$ be the linear operator defined by $Q^t u = -\sgn(t) \sum_{i = 1}^m \b^t(u,w_i) w_i$. Also define the linear operator $P^t : \HH \rightarrow \HH$ by $P^t u  = (I - Q^t) u$, where $I$ is the identity operator. In what follows $\oplus$ denotes the algebraic direct sum of two vector subspaces.

\begin{lemma}[Projection Lemma]
\label{PseudoMon}
Assume $\dim(\Ker(\a)) < \infty$ and $\Ker(\a) \cap \K$ is trivial.
\begin{enumerate}
    \item[(i)] If $u \in \Ker(\a)$ is nonzero, then $\B^t(u) < 0$ whenever $t > T$ and $\B^t(u) > 0$ whenever $t < -T$. Consequently, $P^t$ is a projection operator with $\Range(P^t) = \HH^t$ and induces the decomposition $\HH = \HH^t \oplus \Ker(\a)$ when $|t| > T$.
    
    \item[(ii)] If $T < s \leq t$ then
    \[\B^s(P^s v) \geq \B^s(v) \geq \B^t(v), \quad \text{for each} ~ v \in \HH.\]
\end{enumerate}
\end{lemma}

\begin{proof}[Proof of Lemma \ref{PseudoMon}]
(i) Since $\B$ and $\C$ are weakly (and therefore norm) continuous and $\Ker(\a)$ is finite dimensional, $|\B|$ attains its maximum $M_{|\B|}$ and $\C$ attains its minimum $m_{\C}$ on the unit sphere of $\Ker(\a)$. Moreover, since $\Ker(\a) \cap \K$ is trivial by (C1$'$) we have that $m_{\C} > 0$ by using Cauchy--Schwarz and the definition of $\K = \Ker(\c)$. This shows that $T = \frac{M_{|\B|}}{m_{\C}} + 1$ is well-defined and $\B^{t}$ has a definite sign on $\Ker(\a)$ for all $t$ with $|t| > T$.

It follows that $-\sgn(t) \b^t(\cdot,\cdot)$ is positive definite on $\Ker(\a)$ when $|t| > T$. Thus, $-\sgn(t) \b^t(\cdot,\cdot)$ is an inner-product on $\Ker(\a)$. Choose a $(-\sgn(t) \b^t)$-orthonormal basis $\{w_1, \dots, w_m\}$ of $\Ker(\a)$ and note that $(Q^t)^2 = Q^t$ by a direct calculation. Thus, $Q^t$ is a projection operator with $\Range(Q^t) = \Ker(\a)$. It follows from the definition of $\HH^t$ that $\Ker(Q^t) = \HH^t$. Since $P^t$ is the complementary projection to $Q^t$ we have $\HH = \Range(P^t) \oplus \Ker(P^t) = \HH^t \oplus \Ker(\a)$.

\bigskip

(ii) Using the decomposition $\HH = \HH^s \oplus \Ker(\a)$ each vector in $\HH$ can be written as $v + w$ for some $v \in \HH^s$ and $w \in \Ker(\a)$. Expanding we have
\[\B^s(v+w) = \B^s(v) + 2\b^s(v,w) + \B^s(w) \leq \B^s(v),\]
since $\b^s(v,w) = 0$ and $\B^s \leq 0$ on $\Ker(\a)$. In particular, if $u \in \HH$ is any vector then 
\[\B^s(P^s u) \geq \B^s(u) \geq \B^t(u),\]
where the last inequality is just due to nonnegativity of $\C$.
\end{proof}

\begin{rem}
\label{rem:T}
It follows from Lemma \ref{PseudoMon} that the condition ``$t$ is sufficiently large" from Theorem \ref{SpecConv} and Proposition \ref{NegConv} can be replaced by the quantitative condition $t > T$. In fact, the proof of Lemma \ref{PseudoMon} shows that it suffices to require that $t$ is large enough that $\B^t$ is negative on $\Ker(\a)$ so that $\HH^t$ intersects $\Ker(\a)$ trivially. Analogous statements hold as $t \to -\infty$, that is, for Corollary \ref{cor:NegInfLim}.
\end{rem}

Now we state a result that will establish coercivity of $\A$ on $\K$ and $\HH^t$. Let $\V$ and $\W$ be closed subspaces of a Hilbert space with inner product $\langle \cdot, \cdot \rangle$ and norm $\lVert \cdot \rVert$. Define the quantities
\[\alpha(\V,\W) = \sup\{\langle v,w \rangle : v \in \V, w \in \W, \lVert v\rVert = \lVert w \rVert = 1\} \quad \text{and} \quad \beta(\V,\W) = \sqrt{1 - \alpha(\V,\W)^2},\]
which should be interpreted as the cosine and sine of the angle between $\V$ and $\W$, respectively. In the below theorem $\perp$ denotes the orthogonal complement with respect to the inner product on $\HH$.

\begin{prop}[\protect{\cite[Proposition 1]{Graser}}]
\label{thm:Graser}
Let $\HH$ be a Hilbert space and $\a(\cdot,\cdot) : \HH \times \HH \rightarrow \R$ a continuous symmetric bilinear form with $\dim(\Ker(\a)) < \infty$. If $\A(\cdot)$ is coercive on $\Ker(\a)^{\perp}$ with constant $\gamma > 0$ then $\A(\cdot)$ is coercive with constant $\gamma \cdot \beta(\V,\Ker(\a))^2 > 0$ on each closed subspace $\V$ of $\HH$ with $\V \cap \Ker(\a) = \{0\}$.
\end{prop}

\begin{lemma}[Moving Hilbert Space Coercivity]
\label{MCoercivity}\
If conditions (C1\SP) and (C3) hold then $\A(\cdot)$ is coercive on $\K$ and coercive on $\HH^t$ with a constant that is uniform in $t$ for $|t| > T$.
\end{lemma}

\begin{proof}
Suppose that $|t| > T$. By (C1$'$), $\Ker(\a) \cap \K$ is trivial and by part \emph{(i)} of Projection Lemma \ref{PseudoMon} $\Ker(\a) \cap \HH^t$ is trivial. Since $\dim(\Ker(\a)) < \infty$ and $\A(\cdot)$ is coercive on $\Ker(\a)^{\perp}$ by (C1$'$), Proposition \ref{thm:Graser} implies that $\A$ is coercive on $\K$ and on $\HH^t$ for each $t$. We will show that $\A$ has a uniform coercivity constant on $(-\infty,-T] \cup [T,\infty)$. It is sufficient to show that $\beta(\HH^t,\Ker(\a))$ is lower semicontinuous and $\liminf_{t \to \pm \infty} \beta(\HH^t,\Ker(\a)) > 0$.

First we will show that the supremum defining $\alpha(\HH^t,\Ker(\a))$ is attained for each $t$. Let $(u^n,w^n)$ be an extremizing sequence (with $\lVert u^n \rVert = \lVert v^n \rVert = 1$) and extract a weakly convergent subsequence $u^n \in \HH^t$ and a strongly convergent subsequence $w^n \in \Ker(\a)$ with limit $w$ (using that $\Ker(\a)$ is finite dimensional). If $u^n \rightharpoonup 0$ then $\alpha(\HH^t,\Ker(\a)) = 0$ and $\HH^t \subset \Ker(\a)^{\perp}$ so that $(u,w)$ is an extremizer for any unit norm $u \in \HH^t$. Otherwise, $u^n$ converges weakly to a nonzero $u$ and $(\frac{u}{\lVert u \rVert},w)$ is an extremizer. Thus, the supremum defining $\alpha(\cdot,\cdot)$ is attained.

For each $t$ let $(u_*^t,w_*^t) \in \HH^t \times \Ker(\a)$ be an extremizer so that $\alpha(\HH^t,\Ker(\a)) = \langle u_*^t,w_*^t \rangle$. To see that $t \mapsto \alpha(\HH^t,\Ker(\a))$ is upper semicontinuous let $\{t_n\}_n$ be an arbitrary sequence such that $t_n \to t_{\infty} \in \R$ as $n \to \infty$. Since $u_*^{t_n}$ and $w_*^{t_n}$ have unit norm we can assume (by extracting a subsequence) that $u_*^{t_n} \rightharpoonup u^{t_{\infty}}$ and $w_*^{t_n} \to w^{t_{\infty}}$ for some $(u^{t_{\infty}}, w^{t_{\infty}}) \in \HH^t \times \Ker(\a)$ because $\b^{t_n}(u_*^{t_n},\cdot) \to \b^t(u^{t_{\infty}},\cdot)$ by weak continuity. Thus, 
\[\limsup_{n \to \infty} \alpha(\HH^{t_n},\Ker(\a)) = \lim_{n \to \infty} \langle u_*^{t_n},w_*^{t_n} \rangle = \langle u^{t_{\infty}},w^{t_{\infty}} \rangle \leq \alpha(\HH^t,\Ker(\a)).\]
Hence $t \mapsto \beta(\HH^t,\Ker(\a))$ is lower semicontinuous.

To show $\liminf_{t \to \infty} \beta(\HH^t,\Ker(\a)) > 0$, let $t_n$ be a sequence with $t_n \to \pm \infty$ as $n \to \infty$ and extract a subsequence so that $u_*^{t_n} \rightharpoonup u^{\infty} \in \HH$ and $w_*^{t_n} \rightarrow w^{\infty} \in \Ker(\a)$. Thus, we have 
\[\limsup_{t \to \pm \infty} \alpha(\HH^t,\Ker(\a)) = \limsup_{n \to \infty} \langle u_*^{t_n},w_*^{t_n} \rangle = \langle u^{\infty},w^{\infty} \rangle.\]

Now we show that $\langle u^{\infty}, w^{\infty}\rangle < 1$. Note that $\lVert u^{\infty} \rVert \leq 1$ and $\lVert w^{\infty} \rVert = 1$. If $u^{\infty} \neq w^{\infty}$ we are done so assume that $u^{\infty} = w^{\infty}$. By definition of $\HH^t$ we know that 
\[\b(u_*^{t_n},w) = t_n \c(u_*^{t_n},w), \quad \text{for each} ~ w \in \Ker(\a).\]
By weak continuity, $\b(u_*^{t_n},w)$ is uniformly bounded in $n$ and $\c(u_*^{t_n},w) \to \c(u^{\infty},w)$ so that $\c(u^{\infty},w) = 0$ for each $w \in \Ker(\a)$. Thus, $\c(u^{\infty},u^{\infty}) = 0$ since $u^{\infty} = w^{\infty} \in \Ker(\a)$. Since $\C$ is nonnegative, Cauchy--Schwarz holds so that $\c(u^{\infty},v) \leq \sqrt{\C(u^{\infty})} \sqrt{\C(v)} = 0$ for each $v \in \HH$. This shows $u^{\infty} \in \Ker(\c) = \K$, and since $\Ker(\a) \cap \K$ is trivial we have $u^{\infty} = 0$.

\end{proof}

Although in what follows we work with the eigenvalues of $(\HH^t,\a,\b^t)$, the following lemma shows there is no loss of generality since these eigenvalues coincide with the nonzero eigenvalues of $(\HH,\a,\b^t)$.

\begin{lemma}[Moving eigenequation]
\label{lem:MEE}
Assume that $|t| > T$. If the moving Hilbert space conditions hold and $(\lambda^t,u^t) \in \R \times \HH^t$ satisfies
\begin{equation}
\label{eq:MEE}
\a(u^t,v) = \lambda^t \b^t(u^t,v), \quad \text{for all} ~ v \in \HH^t,
\end{equation}
then the equation holds for all $v \in \HH$. Consequently, $(\HH,\a,\b^t)$ and $(\HH^t,\a,\b^t)$ have the same nonzero eigenvalues, counting multiplicities.
\end{lemma}

\begin{proof}
As soon as $|t| > T$, for each $v \in \HH$ we have the decomposition $v = z + w \in \HH^t \oplus \Ker(\a)$ by the Projection Lemma \ref{PseudoMon}. Thus,
\[\a(u^{t},v) = \a(u^{t},z) = \lambda^{t} \b^t(u^{t},z) = \lambda^t \b^t(u^t,v), ~~ \text{ for each } v \in \HH,\]
so that (\ref{eq:MEE}) holds for all $v \in \HH$. 

Consequently, each eigenpair of $(\HH^t,\a,\b^t)$ is also an eigenpair of $(\HH,\a,\b^t)$. Conversely, for each eigenpair $(\lambda^t,u^t)$ of $(\HH,\a,\b^t)$ with $\lambda^t \neq 0$ we have $u^t \in \HH^t$ by choosing $v = w \in \Ker(\a)$, so that $(\lambda^t,u^t)$ is an eigenpair of $(\HH^t,\a,\b^t)$. 
\end{proof}


\subsection*{Variational characterizations}


We first state an inductive characterization of the eigenvalues due to Auchmuty \cite{Auch}. Suppose $\a$ and $\b$ are symmetric bilinear forms on an arbitrary Hilbert space $\HH$ and that the problem $(\HH,\a,\b)$ has $j-1$ $\a$-orthonormal eigenvectors whose span is denoted by $\U_{j-1}$. Let 
\begin{equation}
\label{eq:beta}
\beta_j = \sup\{\B(u) : u \in \U_{j-1}^{\perp_{\a}} ~ \text{and} ~ \A(u) = 1\} ~ \text{for} ~ j \geq 2,
\end{equation}
where $\perp_{\a}$ denotes the $\a$-orthogonal complement. For $j = 1$ we let $\beta_1 = \sup \{\B(u) : u \in \HH, \A(u) = 1\}$.
\begin{theorem}[Existence; \protect{\cite[Theorems 3.1 and 4.2]{Auch}}]
\label{thm:AuchExist}
Assume (C1)--(C3) hold. If $\B(u) > 0$ for some $u \in \HH$ then $\beta_1 > 0$, $\lambda_1(\HH,\a,\b)$ exists and equals $\beta_1^{-1}$, and there is an eigenvector of $(\HH,\a,\b)$ at which the supremum defining $\beta_1$ is attained. If $j \geq 1$ and $\U_{j-1}$ are as above, then either:
\begin{enumerate}
\item[(i)] $\beta_j = 0$ and $(\HH,\a,\b)$ has exactly $j-1$ positive eigenvalues and $\B \leq 0$ on $\U_{j-1}^{\perp_{\a}}$, or
\item[(ii)] $\beta_j > 0$ and $\lambda_j(\HH,\a,\b)$ exists, equals $\beta_j^{-1}$, and has an eigenvector at which the supremum defining $\beta_j$ is attained. 
\end{enumerate}
Thus, the positive eigenvalues of $(\HH,\a,\b)$ have the form 
\[0 < \lambda_1 \leq \lambda_2 \leq \dots,\]
where the number of positive eigenvalues may be zero, finite, or infinite. 

\end{theorem}

The following two theorems due to Auchmuty are crucial tools for proving monotonicity of eigenvalues and our stability result Proposition \ref{prop:StabSpec}. The theorem below is a slight strengthening of \cite[Theorem 5.1]{Auch}.

\begin{theorem}[Variational Characterization; \protect{\cite[Theorem 5.1]{Auch}}]
\label{VC}
Assume that (C1)--(C3) hold for $(\HH,\a,\b)$ and let $i \geq 1$. If $\lambda_i$ is a positive eigenvalue of $(\HH,\a,\b)$ then
\[\frac{1}{\lambda_i} = \sup_{\Ss_i \subset \HH} \inf \{\B(u) : u \in \Ss_i \text{ and } \A(u) =1\} > 0,\]
where $\Ss_i$ ranges over all $i$-dimensional subspaces of $\HH$. Conversely, if the above supremum is positive then $\lambda_i$ exists and is positive.
\end{theorem}

\begin{proof}
The variational characterization itself is precisely \cite[Theorem 5.1]{Auch}. To see the converse statement we proceed by induction. Suppose that $i = 1$. Observe that if the supremum is positive then $\B$ is positive somewhere on the $\a$-unit sphere of $\HH$, and so $\lambda_1$ exists by Theorem \ref{thm:AuchExist}. 

Suppose that the result holds for $i-1$. Observe that if the supremum is positive then $\B$ is positive on the $\a$-unit sphere of some $i$-dimensional subspace $\widehat{\Ss_i}$. Let $\U_{i-1}$ denote the $(i-1)$-dimensional subspace spanned by the first $i-1$ eigenvectors of the problem $(\HH,\a,\b)$, which exists by the inductive hypothesis and Theorem \ref{thm:AuchExist}. By dimension counting, $\widehat{\Ss_i} \cap \U_{i-1}^{\perp_{\a}}$ is nontrivial. Thus, $\B$ is positive somewhere on $\U_{i-1}^{\perp_{\a}}$ so that $\lambda_i$ exists by Theorem \ref{thm:AuchExist}.
\end{proof}

Since $\A$ is coercive on $\HH^t$ by Lemma \ref{MCoercivity}, Theorem \ref{VC} may be applied to the problem $(\HH^t,\a,\b^t)$ for each $t$ with $|t| > T$. To prove bounds on and monotonicity of eigenvalues of $(\HH^t,\a,\b^t)$ it will be useful to expand the supremum in the variational characterization to a collection of subspaces that is independent of $t$.

\begin{theorem}[Moving Hilbert Space Variational Characterization]
\label{MVC}
Assume that the moving Hilbert space conditions hold, $t > T$, and $i \geq 1$. If $\lambda_i^t = \lambda_i(\HH^t,\a,\b^t)$ exists then 
\[\frac{1}{\lambda_i^t} = \hspace{-4mm} \sup_{\Ss_i : \Ss_i \cap \Ker(\a) = \{0\}} \hspace{-6mm} \inf \{\B^t(u) : u \in \Ss_i \text{ and } \A(u) = 1\} > 0,\]
where $\Ss_i$ ranges over $i$-dimensional subspaces of $\HH$. Conversely, if the supremum above is positive for some $i$ then $\lambda_i^t$ exists and is positive.
\end{theorem}

\begin{proof}

By the original variational characterization in Theorem \ref{VC}, it is enough to show that 
\begin{equation}
\label{eq:VCEquality}
\sup_{\Ss_i \subset \HH^t} \inf \{\B^t(u) : u \in \Ss_i \text{ and } \A(u) =1\} = \hspace{-2mm} \sup_{\Ss_i : \Ss_i \cap \Ker(\a) = \{0\}} \hspace{-6mm} \inf \{\B^t(u) : u \in \Ss_i \text{ and } \A(u) = 1\}.
\end{equation}

The left side is at most the right since if $\Ss_i \subset \HH^t$ then $\Ss_i \cap \Ker(\a) = \{0\}$ because $\HH^t \cap \Ker(\a) = \{0\}$ for $t > T$ by Lemma \ref{PseudoMon}.

To see the opposite inequality let $\widehat{\Ss_i} \subset \HH$ be an $i$-dimensional subspace with the property that $\widehat{\Ss_i} \cap \Ker(\a)$ is trivial. Recall that $P^t$ is a projection onto $\HH^t$ with $\Ker(P^t) = \Ker(\a)$. Thus, the subspace $P^t \widehat{\Ss_i}$ is also $i$-dimensional since $\widehat{\Ss_i} \cap \Ker(P^t)$ is trivial. Therefore, $P^t \widehat{\Ss_i}$ is a valid trial subspace for the left side of (\ref{eq:VCEquality}) so that the left side is larger than
\begin{align}
&\inf\{\B^t(u) : u \in P^t \widehat{\Ss_i}, \A(u) = 1\}\notag\\
= &\inf\{\B^t(P^t v) : v \in \widehat{\Ss_i}, \A(v) = 1\}\notag\\
\geq &\inf\{\B^t(v) : v \in \widehat{\Ss_i}, \A(v) = 1\}\label{eq:MVC}.
\end{align}
In the second to last step we used that $\A(P^t v) = \A(v)$ by writing $v$ into the decomposition $\HH = \HH^t \oplus \Ker(\a)$ and using that $P^t$ is a projection on $\HH^t$. In the final step we used that $\B^t \circ P^t \geq \B^t$ from the Projection Lemma \ref{PseudoMon}. Taking a supremum over all such $\widehat \Ss_i$ in (\ref{eq:MVC}) shows that the two suprema are equal. 

If (\ref{eq:MVC}) is positive then the left side of (\ref{eq:VCEquality}) is also positive by the above calculation and so the converse statement in the theorem holds. 

\end{proof}

\subsection*{Proof of Existence of spectrum: Theorem \ref{thm:ExistSpec}}

The decomposition in Theorem \ref{thm:ExistSpec} was originally stated in \cite{Auch}, with an erroneous definition of $\U_{\infty}$ that Auchmuty later corrected in \cite{AuchWeb}. We reprove the corrected result below. 

\begin{proof}

The existence of the positive eigenvalues of $(\HH,\a,\b)$ follows from the Existence Theorem \ref{thm:AuchExist}. The existence of the negative spectrum follows from applying Theorem \ref{thm:AuchExist} to the problem $(\HH,\a,-\b)$ and observing that $\lambda_{-j}(\HH,\a,\b) = -\lambda_j(\HH,\a,-\b)$.

To see the decomposition result let $\perp_{\a}$ denote the $\a$-orthogonal complement and recall that $\U_+$ and $\U_-$ are the closed spans of the eigenvectors with positive ($+$) and negative ($-$) eigenvalues, respectively. Observe that by the eigenequation eigenvectors with distinct eigenvalues are $\a$-orthogonal so that $\U_+$ and $\U_-$ are $\a$-orthogonal and intersect trivially. We proceed by showing $\Ker(\b) = (\U_+ \oplus_{\a} \U_-)^{\perp_{\a}}$. For the forward inclusion let $u \in \Ker(\b)$ and take $v = u$ in the eigenequation so that
\[\a(u_{\pm j},u) = \lambda_{\pm j} \b(u_{\pm j},u), \quad \text{for each} ~ j \geq 1,\]
where $u_{\pm j}$ are eigenvectors. Since $u \in \Ker(\b)$ we have $\a(u_{\pm j},u) = 0$ for each $j \geq 1$ so that $u \in (\U_+ \oplus_{\a} \U_-)^{\perp_{\a}}$. 

To see the ``$\supset$" inclusion, let $u \in (\U_+ \oplus_{\a} \U_-)^{\perp_{\a}}$ and $v \in \HH$. When $v = u_{\pm j}$ is an eigenvector, $\b(u,v) = 0$ by the eigenequation since $u \perp_{\a} u_{\pm j}$. By linearity and weak continuity of $\b(\cdot,\cdot)$ we have $\b(u,v) = 0$ for each $v \in \U_+ \oplus_{\a} \U_-$. Observe that $\B \leq 0$ on $\U_+^{\perp_{\a}}$ and $\B \geq 0$ on $\U_-^{\perp_{\a}}$ due to the Existence Theorem \ref{thm:AuchExist} so that $\B = 0$ on $(\U_+ \oplus_{\a} \U_-)^{\perp_{\a}}$. If $v \in (\U_+ \oplus_{\a} \U_-)^{\perp_{\a}}$ then $u + v \in (\U_+ \oplus_{\a} \U_-)^{\perp_{\a}}$ so that
\[0 = \B(u+v) = \B(u) + \B(v) + 2\b(u,v) = 2\b(u,v).\]
Thus, $\b(u,v) = 0$ for each $v \in \HH$ so that $u \in \Ker(\b)$.
\end{proof}

\section{\textbf{Monotonicity and Convergence Lemmas}}

Recall that $J_+$ is the number of positive eigenvalues of $(\K,\a,\b)$ and let $\lambda^t_j = \lambda_j(\HH,\a,\b^t)$. To prove our convergence results, first we show that when $j \leq J_+$ each $\lambda_j^t$ has a limit as $t \to \infty$. By using the eigenequation this convergence will help us show that the limit of each eigenvector is in $\K$. 

\begin{lemma}[Eigenvalue Monotonicity and Limit]
\label{LimExists}
If the fixed or moving Hilbert space conditions hold (and $t > T$ in the moving case) and $j \geq 1$ then $\lambda^t_j$ is increasing in $t$ whenever it exists. If, in addition, $1 \leq j \leq J_+$ then:
\begin{enumerate}
\item[(i)] $\lambda^t_j \leq \lambda_j(\K,\a,\b)$,
\item[(ii)] $\lim_{t \to \infty} \lambda^t_j$ and $\lim_{t \nearrow s} \lambda_j^t$ exist and are at most $\lambda_j(\K,\a,\b)$ for each $s \in \R$ ($s > T$ in the moving Hilbert space case).
\end{enumerate}

\end{lemma}

\begin{proof}

Observe that by Proposition \ref{prop:StabSpec}, if $\lambda_j^t$ exists for some $t$ then it exists an open interval around $t$ so it makes sense to say that $\lambda_j^t$ is increasing on this set. Since $\B^t$ is decreasing on $\HH$, the variational characterizations in Theorems \ref{VC} and \ref{MVC} show that $\lambda^t_j$ is increasing in $t$ (for $t > T$ in the moving Hilbert space case). In the fixed Hilbert space case, using the variational characterization and that $\C = 0$ on $\K$ shows 
\begin{align}
\frac{1}{\lambda^t_j} &= \sup_{\Ss_j \subset \HH} \inf \{\B^t(u) : u \in \Ss_j \text{ and } \A(u) = 1\}\notag\\
&\geq \sup_{\Ss_j \subset \K} \inf \{\B^t(u) : u \in \Ss_j \text{ and } \A(u) = 1\}\notag\\
&= \sup_{\Ss_j \subset \K} \inf \{\B^0(u) : u \in \Ss_j \text{ and } \A(u) = 1\}\notag\\
&= \frac{1}{\lambda_j(\K,\a,\b)}.\notag
\end{align}

The same holds in the moving Hilbert space case by imposing that $\Ss_j \cap \Ker(\a)$ is trivial in the first supremum and noting that $\K \cap \Ker(\a)$ is also trivial.

Taking reciprocals we see that $\lambda^t_j$ is bounded from above by $\lambda_j(\K,\a,\b)$ and is non-decreasing so the limit exists and is at most $\lambda_j(\K,\a,\b)$. 

\end{proof}

Note that when the fixed Hilbert space conditions hold, $\a(\cdot,\cdot)$ is an inner product on $\HH$ that induces a norm equivalent to $\lVert \cdot \rVert$ . When the moving Hilbert space conditions hold it only induces a seminorm on $\HH$. The following lemma summarizes some facts that hold for $\a(\cdot,\cdot)$ in either the fixed or moving Hilbert space case. In what follows ``weakly convergent" and ``$\rightharpoonup$" will mean weakly convergent with respect to the inner product on $\HH$. Recall the lower triangle inequality for $\a(\cdot,\cdot)$: $\left| \sqrt{\A(u)} - \sqrt{\A(v)} \right| \leq \sqrt{\A(u - v)}$ for each $u,v \in \HH$.
\begin{lemma}[$\a$-Lemma]
\label{lem:ALem}
Assume that $\a(\cdot,\cdot): \HH \times \HH \to \R$ is a symmetric bilinear form. If $\a(\cdot,\cdot)$ satisfies either (C1) or (C1\SP), and (C2), then: 
\begin{enumerate}
    \item[(i)] Cauchy--Schwarz and the lower triangle inequality hold for $\a(\cdot,\cdot)$;
    \item[(ii)] If $v^n \rightharpoonup v$  then $\a(v^n,u) \to \a(v,u)$ for each $u \in \HH$;
    \item[(iii)] If $v^n \rightharpoonup v$ and $\A(v^n) \to \A(v)$ then $\A(v^n - v) \to 0$;
    \item[(iv)] If $\a(v_1^n,v_2^n) = 0$ for each $n$ and $v^n_1$ and $v^n_2$ converge in the $\sqrt{\A(\cdot)}$ (semi)norm to $v_1$ and $v_2$ respectively, then $A(v_1,v_2) = 0$;
    \item[(v)]$\sqrt{\A(\cdot)}$ is weakly sequentially lower semi-continuous on $(\HH,\lVert \cdot \rVert )$.
\end{enumerate}
\end{lemma}

\begin{proof}
To see (ii) simply note that $v \mapsto \a(v,u)$ is a continuous linear functional on $\HH$ for each $u \in \HH$, by assumption (C2). 

Statement (i) can be proved by appealing to the fact that $\sqrt{\A(\cdot)}$ is a norm on $\Ker(\a)^{\perp}$. Statements (iii) and (iv) can be proved as if $\a(\cdot,\cdot)$ were an inner-product on $\HH$. 

To see (v) note that $\sqrt{\A(\cdot)}$ is a norm on $\Ker(\a)^{\perp}$ (where $\perp$ is with respect to the $\HH$-inner product) so that it is weakly sequentially lower semi-continuous (w.s.l.s.c.)\ on $(\Ker(\a)^{\perp},\sqrt{\A(\cdot)})$. Since $\sqrt{\A(\cdot)}$ and $\lVert \cdot \rVert$  are equivalent norms they have the same set of continuous linear functionals, and therefore, the same set of weakly convergent sequences. Hence $\sqrt{\A(\cdot)}$ is w.s.l.s.c.\ on $(\Ker(\a)^{\perp},\lVert \cdot \rVert )$, and thus on all of $\HH$, as follows from using $\HH = \Ker(\a)^{\perp} \oplus \Ker(\a)$ and the definition of w.s.l.s.c.
\end{proof}

In what follows let $\lambda_j^{t_{\infty}} \in (0,+\infty]$ denote $\lim_{n \to \infty} \lambda_j^{t_n}$ for $j \geq 1$, where $\{t_n\}_n \subset \R$ is a sequence such that $t_n \nearrow t_{\infty} \in (-\infty,+\infty]$ as $n \to \infty$. In the moving Hilbert space case we impose that $t_{\infty} > T$.

\begin{lemma}[Strong Convergence]
\label{lem:Strong}
Assume the fixed or moving Hilbert space conditions hold, $j \geq 1$,  and $\{u_j^{t_n}\}_n$ is a sequence of $\HH$-normalized eigenvectors of $(\HH,\a,\b^{t_n})$ with weak limit $u_j^{t_{\infty}} \in \HH$ and positive eigenvalues $\lambda_j^{t_n}$. If $\lambda_j^{t_{\infty}} < \infty$ then:
\begin{enumerate}
    \item[(i)] $\A(u_j^{t_n} - u_j^{t_\infty}) \to 0$, $\lVert u_j^{t_n} - u_j^{t_\infty} \rVert \to 0$, and $u_j^{t_\infty} \neq 0$.
    \item[(ii)] If $t_{\infty} < \infty$ then
\[\a(u_j^{t_\infty},v) = \lambda_j^{t_\infty}\b^{t_{\infty}}(u_j^{t_\infty},v), \quad \text{for all}~ v \in \HH,\]
and if $t_{\infty} = \infty$ then $u^{t_{\infty}}_j \in \K$ and
\[\a(u_j^{t_\infty},v) = \lambda_j^{t_\infty}\b(u_j^{t_\infty},v), \quad \text{for all}~ v \in \K.\]
\end{enumerate}
On the other hand, if $\lambda_j^{t_{\infty}} = \infty$ and $t_{\infty} = \infty$ then:
\begin{enumerate}[resume]
    \item[(iii)] $u_j^{t_\infty} \in \Ker(\b|_{\K \times \K}) \subset \K$.
\end{enumerate}

\end{lemma}

\noindent Note that part (iii) of the above lemma does not claim that $u^{t_n}_j$ converges strongly, so $u^{t_{\infty}}_j$ may be zero. 

\begin{proof} Since the index $j$ on the eigenvalue and eigenvector is fixed we will drop it for notational ease. When $t_{\infty} = \infty$ we will write $u^{\infty} = u^{t_\infty}$ and $\lambda^{\infty} = \lambda^{t_\infty}$.

\bigskip

(i)/(ii) \textbf{Case 1: $(t_{\infty} = \infty)$:} To prove (i) and (ii), assume $\lambda^{\infty} < \infty$. By rearranging the eigenequation, in either the fixed or moving Hilbert space case  (Lemma \ref{lem:MEE}) for every $t_n > T$ we have
\[t_n \c(u^{t_n},v) = \b(u^{t_n},v) - \frac{1}{\lambda^{t_n}} \a(u^{t_n},v), \quad \text{for all} \quad v \in \HH.\]

The right side is bounded since $\a(u^{t_n},v),\b(u^{t_n},v),$ and $1/\lambda^{t_n}$ are each bounded by weak convergence of $u^{t_n}$, part (ii) of the $\a$-Lemma \ref{lem:ALem}, and Monotonicity Lemma \ref{LimExists}. This proves the ``draining estimate"
\begin{equation}
\label{eq:DrainInq}
|\c(u^{t_n},v)| \lesssim_v \frac{1}{t_n}.
\end{equation}
Taking the limit as $n \to \infty$ and using weak continuity of $\c(\cdot,\cdot)$, we have that $\c(u^{\infty},v) = \lim_{n \rightarrow \infty} \c(u^{t_n},v) = 0$ for each $v \in \HH$. This shows that $u^{\infty} \in \Ker(\c) = \K$.

Examining the eigenequation again we have 
\begin{equation}
\label{eq:EE}
\a(u^{t_n},v) = \lambda^{t_n} \b^{t_n}(u^{t_n},v), ~~ \text{ for each } v \in \HH,
\end{equation}
by the moving eigenequation Lemma \ref{lem:MEE}. Choosing $v \in \K = \Ker(\c)$, the eigenequation reduces to
\begin{equation}
\label{eq:EEResK}
\a(u^{t_n},v) = \lambda^{t_n} \b(u^{t_n},v), ~~ \text{ for each } v \in \K.
\end{equation}
Recall that $\lambda^{\infty} < \infty$ by assumption. By weak convergence, taking $n \rightarrow \infty$ we obtain
\begin{equation}
\label{eq:LimEigEq1}
\a(u^{\infty},v) = \lambda^{\infty} \b(u^{\infty},v), ~ \text{ for each } ~ v \in \K.
\end{equation}

To see strong convergence holds in the (semi)norm $\sqrt{\A(\cdot)}$, recall that the $\a$-Lemma \ref{lem:ALem} says:
\begin{equation}
\label{eq:ConvImp}
u^{t_n} \rightharpoonup u^{t_\infty} ~ \text{and} ~ \A(u^{t_n}) \rightarrow \A(u^{t_\infty}) \implies \A(u^{t_n} - u^{t_\infty}) \rightarrow 0.
\end{equation}
We proceed by checking $\A(u^{t_n}) \rightarrow \A(u^{\infty})$. Since $\sqrt{\A}$ is weakly sequentially lower semi-continuous by the $\a$-Lemma \ref{lem:ALem}, we know
\begin{equation}
\label{eq:wlsc}
\liminf_{n \rightarrow \infty} \A(u^{t_n}) \geq \A(u^{\infty}).
\end{equation}

On the other hand, using the eigenequation and nonnegativity of $\lambda^{t_n} t_n\C$ we have that
\[\A(u^{t_n}) = \lambda^{t_n} [\B(u^{t_n}) - t_n \C(u^{t_n})] \leq \lambda^{t_n}\B(u^{t_n}).\]
Taking the limsup of both sides, then using weak continuity of $\B(\cdot)$ and (\ref{eq:LimEigEq1}) with $v = u_j^{\infty} \in \K$ we have that
\begin{equation}
\label{eq:LimSup}
\limsup_{n \rightarrow \infty} \A(u^{t_n}) \leq \lim_{n \rightarrow \infty} \lambda^{t_n}\B(u^{t_n}) = \lambda^{\infty} \B(u^{\infty}) = \A(u^{\infty}).
\end{equation}

Putting (\ref{eq:wlsc}) and (\ref{eq:LimSup}) together implies $\A(u^{t_n}) \to \A(u^{\infty})$, and now (\ref{eq:ConvImp}) gives that $\A(u^{t_n} - u^{\infty}) \to 0$. The remaining statements of part (i) will be proven at the end of Case 2 below. 

\bigskip

\textbf{Case 2: $(t_{\infty} < \infty)$:} Taking $n \to \infty$ in (\ref{eq:EE}) shows that $\lambda_j^{t_{\infty}}$ and $u_j^{t_{\infty}}$ satisfy the the first equation in (ii) by weak continuity. To see that strong convergence holds in the (semi)norm $\sqrt{\A(\cdot)}$ when $t_{\infty} < \infty$, we again use (\ref{eq:ConvImp}) and proceed by checking $\A(u^{t_n}) \to \A(u^{t_{\infty}})$. Since (\ref{eq:wlsc}) stil holds, we only have to check $\limsup_{n \rightarrow \infty} \A(u^{t_n}) \leq \A(u^{t_\infty})$. Let $v = u^{t_n}$ in (\ref{eq:EE}) and take the limsup of both sides so that
\begin{equation*}
\limsup_{n \to \infty} \A(u^{t_n}) = \limsup_{n \to \infty} \lambda^{t_n} \B^{t_n}(u^{t_n}) = \lambda^{t_{\infty}} \B^{t_{\infty}}(u^{t_{\infty}}),
\end{equation*}
by weak continuity of $\B^t(\cdot)$.
After taking $v = u^{t_{\infty}}$ in (\ref{eq:EE}), sending $n \to \infty$, and using weak convergence we have
\[\A(u^{t_{\infty}}) = \lim_{n \to \infty} \a(u^{t_n},u^{t_{\infty}}) = \lim_{n \to \infty} \lambda^{t_n} \b^{t_n}(u^{t_n},u^{t_{\infty}}) = \lambda^{t_{\infty}} \B^{t_{\infty}}(u^{t_{\infty}}),\]
so that $\limsup_{n \to \infty} \A(u^{t_n}) = \A(u^{t_{\infty}})$ and hence $\A(u^{t_n}) \to \A(u^{t_{\infty}})$. The remainder of the proof is identical to the case when $t_{\infty} = \infty$.

\bigskip

It remains to show that $\lVert u^{t_n} - u^{t_{\infty}} \rVert \to 0$ as $n \to \infty$ and $u^{t_{\infty}} \neq 0$. Uniform coercivity implies that $\lVert u^{t_n} - u^{t_{\infty}} \rVert^2  \lesssim \A(u^{t_n} - u^{t_{\infty}}) \to 0$ as $n \to \infty$. Since $u^{t_n}$ is $\HH$-normalized this shows that $\lVert u^{t_{\infty}} \rVert = \lim_{n \to \infty} \lVert u^{t_n} \rVert = 1$ so that $u^{t_{\infty}} \neq 0$.

\bigskip

(iii) Now assume that $\lambda^{\infty} = \infty$ and $t_{\infty} = \infty$. By using (\ref{eq:EEResK}) and that $\a(u^{t_n},v)$ is bounded in $n$ for each $v \in \K$ we have
\[|\b(u^{t_n},v)| \lesssim_v \frac{1}{\lambda^{t_n}} \to 0,\]
so that by weak continuity $\b(u^{\infty},v) = 0$ and hence $u^{\infty} \in \Ker(\b|_{\K \times \K})$.

\end{proof}

\begin{rem}
\label{rem:CDrain}
The draining estimate (\ref{eq:DrainInq}) is what establishes $(\K,\a,\b)$ as a candidate for the limiting problem since it shows that limits of eigenvectors are in $\K$. Moreover, when the fixed Hilbert space condition holds and $\B$ is positive somewhere on $\HH$, a similar argument refines the $\C$-draining bound (\ref{eq:DrainInq}) to the more explicit bound: $\C(u^{t_n}) \leq \mu/t_n$, where $\mu = [\lambda_1(\HH,\a,\b)]^{-1} > 0$.
\end{rem}

In the following lemma let $t_n \to \infty$ as $n \to \infty$ and $\lambda_{-j}^{\infty} = \lim_{n \to \infty} \lambda_{-j}^{t_n}$.

\begin{lemma}[Convergence of eigenvectors with negative eigenvalue]
\label{lem:ConNEv}
Assume the fixed or moving Hilbert space conditions hold, $j \geq 1$, and  $\{u_{-j}^{t_n}\}_n$ is a sequence of weakly convergent eigenvectors of $(\HH,\a,\b^{t_n})$ with weak limit $u_{-j}^{\infty} \in \HH$ and negative eigenvalues $\lambda_{-j}^{t_n}$. 
\begin{enumerate}
    \item[(i)] If $\lambda_{-j}^{\infty} < 0$ then $u_{-j}^{\infty} \in \K$ and
    \begin{equation}
    \label{eq:NegEE}
    \a(u_{-j}^{\infty},v) = \lambda_{-j}^{\infty} \b(u_{-j}^{\infty},v), \quad \text{for all}~ v \in \K.
    \end{equation}
    \item[(ii)] If $\lambda_{-j}^{\infty} = 0$ then $u_{-j}^{\infty} \in \Ker(\c)^{\perp_{\a}}$.
\end{enumerate}
\end{lemma}

\begin{proof}
(i) Since $\lim_{n \to \infty} \lambda_{-j}^{t_n} < 0$ we know that $1/\lambda_{-j}^{t_n}$ is bounded so the same argument that proves the draining estimate (\ref{eq:DrainInq}) shows that $u_{-j}^{\infty} \in \K$. That $\lambda_{-j}^{\infty}$ and $u_{-j}^{\infty}$ satisfy (\ref{eq:NegEE}) follows from the argument that proves the $t_{\infty} = \infty$ case of (ii) in the proof of Lemma \ref{lem:Strong}.

(ii) By choosing $v \in \Ker(\c)$ in the eigenequation for $u_{-j}^{t_n}$ we find that 
\[\a(u_{-j}^{t_n},v) = \lambda_{-j}^{t_n} \b(u_{-j}^{t_n},v), \quad \text{for all} ~ v \in \Ker(\c).\]
Since the $\b(u_{-j}^{t_n},v)$ is bounded in $n$ and $\lambda_{-j}^{t_n} \to 0$ as $n \to \infty$ we have $\a(u_{-j}^{\infty},v) = 0$ for each $v \in \Ker(\c)$ by weak convergence. This shows that $u_{-j}^{\infty}$ is in the $\a$-orthogonal complement of $\Ker(\c)$.
\end{proof}

\section{\textbf{Proofs of Main Results}}

We will prove our convergence and continuity results. The following Lemma will show that the limit of $\lambda_j(\HH^t,\a,\b^t)$ as $t \to \infty$ is an eigenvalue of $(\K,\a,\b)$.

\begin{lemma}[Containment of the Spectrum]
\label{SpecCont}
Assume the fixed or moving Hilbert space conditions hold and $j \geq 1$. If $\{u^t_j\}_{t > 0}$ is a collection of $\HH$-normalized eigenvectors with eigenvalues $\lambda^t_j$ such that $\lim_{t \to \infty} \lambda^t_j < \infty$ then:
\begin{enumerate}
\item[(i)] a subsequence $u^{t_n}$ exists that converges in norm to some $u^{\infty} \in \K$, and 
\item[(ii)] $(u^{\infty}, \lambda^{\infty})$ is an eigenpair of $(\K,\a,\b)$, where $\lambda^{\infty} = \lim_{t \to \infty} \lambda^t$. 
\end{enumerate}

\end{lemma}

\begin{proof}
Since $\{u^t\}_{t > 0}$ is bounded we can extract a weakly convergent subsequence $\{u^{t_n}\}_n$. By the Strong Convergence Lemma \ref{lem:Strong} this subsequence converges strongly with a nonzero limit $u^{\infty} \in \K$ and $u^{\infty}$ satisfies
\begin{equation}
\label{eq:LimEigEq}
\a(u^{\infty},v) = \lambda^{\infty} \b(u^{\infty},v) ~~ \text{ for each } v \in \K.
\end{equation}
Since $u^{\infty}$ is nonzero, it is a genuine eigenvector of equation (\ref{eq:LimEigEq}) and $(u^{\infty}, \lambda^{\infty})$ is an eigenpair of $(\K,\a,\b)$. 
\end{proof}

\subsection*{Proof of Theorem \ref{SpecConv}: Convergence of Spectrum as \texorpdfstring{$t \to + \infty$}{Lg}}

\begin{proof}
Part (i). ($j \leq J_+$). Let $\lambda_j^{\infty} = \lim_{t \rightarrow \infty} \lambda_j^{t}$ for each $j = 1, \dots, J_+$. Recall that the limit exists by the Monotonicity Lemma \ref{LimExists}. 

\bigskip

\textbf{Base Case: $(j = 1)$:} By Lemma \ref{LimExists} we have $\lambda_1^t \leq \lambda_1(\K)$ so that $\lambda_1^{\infty} \leq \lambda_1(\K)$. By Lemma \ref{SpecCont} we know $\lambda_1^{\infty}$ is an eigenvalue of $(\K,\a,\b)$ so we must have that $\lambda_1^{\infty} = \lambda_1(\K)$.

\bigskip

In preparation for the inductive step, we generalize the Strong Convergence Lemma to give strong convergence of a collection of eigenvectors. In what follows, $u$ and $v$ are said to be \emph{$\a$-orthonormal} if $\a(u,v) = 0$ and $\A(u) = \A(v) = 1$. Suppose $\{u_1^t, \dots, u_{j+1}^t\}$ is a set of $j+1$ $\a$-orthonormal eigenvectors of $(\HH,\a,\b^t)$ generated by the inductive procedure defining $\beta_{j+1}$ in (\ref{eq:beta}) with eigenvalues $\lambda_i^t$ such that $\lim_{t \to \infty} \lambda_i^t < \infty$ for $i = 1, \dots, j+1$.

By Lemma \ref{lem:Strong} part (i), we can iteratively extract subsequences from $\{u_i^t\}_{t > 0}$ for $i = 1, \dots, j+1$, to produce a common subsequence such that each eigenvector converges strongly to an eigenvector $u_i^{t_\infty} \in \HH$ along this common subsequence. Call the common subsequence $t_n$, let 
\begin{equation}
\label{eq:Ut}
\mathscr{U}_j^{t_n} = \{u_1^{t_n}, \dots, u_j^{t_n}\}, \quad \text{for each} ~ n \geq 1,
\end{equation}
and denote the set of limits by $\mathscr{U}^{t_\infty}_j$. Since $\mathscr{U}_j^{t_n} \cup \{u_{j+1}^{t_n}\}$ is an $\a$-orthonormal set for each $n$, by the $\a$-Lemma \ref{lem:ALem} part (iv) we know $\mathscr{U}^{t_\infty}_j \cup \{u_{j+1}^{t_\infty}\}$ is still an $\a$-orthonormal set. When $t_{\infty} = \infty$ we will write $\mathscr{U}^{\infty}_j$ for $\mathscr{U}^{t_\infty}_j$ and $u_i^{\infty}$ for $u_i^{t_{\infty}}$.

\bigskip

\textbf{Inductive Step: $(1 \leq j < J_+)$:} Assume $J_+ \geq 2$, $1 \leq j < J_+$, and that $\lambda_i^t \nearrow \lambda_i(\K)$ for each $i = 1, \dots, j$ as $t \to \infty$. We will show that $\lambda_{j+1}^t \nearrow \lambda_{j+1}(\K)$. By Lemma \ref{LimExists},
\[\lambda_j^{t} \leq \lambda_{j+1}^{t} \leq \lambda_{j+1}(\K).\]
Taking limits and using the inductive assumption $\lambda_j^{\infty} = \lambda_j(\K)$ we obtain
\begin{equation*}
\lambda_j(\K) \leq \lambda_{j+1}^{\infty} \leq \lambda_{j+1}(\K).
\end{equation*}
If $\lambda_j(\K) = \lambda_{j+1}(\K)$ then we are done, so assume they are different. Lemma \ref{SpecCont} implies $\lambda_{j+1}^{\infty}$ equals either $\lambda_j(\K)$ or $\lambda_{j+1}(\K)$. Now we show that $\lambda_{j+1}^{\infty} \neq \lambda_j(\K)$.

Let $\mathscr{U}^{t_n}_j$ and $u_{j+1}^{t_n}$ be constructed as in (\ref{eq:Ut}) and $E_j \subset \K$ denote the eigenspace associated to $\lambda_j(\K)$. By the inductive assumption and that $\lambda_j(\K) \neq \lambda_{j+1}(\K)$, we know $E_j \subset \Span \mathscr{U}_j^{\infty}$. Since $u_{j+1}^{\infty}$ is $\a$-orthogonal to $\Span \mathscr{U}_j^{\infty}$, it is also $\a$-orthogonal to $E_j$. Thus, $u_{j+1}^{\infty} \not\in E_j$ and $\lambda_{j+1}^{\infty} \neq \lambda_j(\K)$ so that $\lambda_{j+1}^{\infty} = \lambda_{j+1}(\K)$. 

\bigskip

Part (ii). \textbf{$(j > J_+)$:} Now suppose $J_+ < \infty$, $j > J_+$, and that $\lambda_j^{t_*}$ exists with an $\a$-normalized eigenvector $u_j^{t_*}$. Let
\[t_j = \inf\{t > t_* : \lambda_j^t ~ \text{does not exist}\}.\]

Note that $t_* < t_j$ by Proposition \ref{prop:StabSpec}. Our goal is to show that $\lambda_j^t \nearrow + \infty$ as $t \nearrow t_j$. Let $\{u_1^{t_n}, \dots, u_j^{t_n}\}$ be as in (\ref{eq:Ut}) for an aribitrary sequence $t_n  \nearrow t_j$.

\bigskip

\textbf{Case 1:} Suppose that $t_j = \infty$ so that $\lambda_j^t$ exists for all large $t$. Since $\mathscr{U}_j^{\infty}$ is an $\a$-orthonormal set, by the inductive step $\{u_1^{\infty}, \dots, u_{J_+}^{\infty}\}$ is a maximal linearly independent set of eigenvectors of $(\K,\a,\b)$ with positive eigenvalues. Now suppose, towards a contradiction, that $\lambda_j^{\infty} < \infty$. Then $u_j^{\infty}$ is $\a$-orthogonal to $\Span\{u_1^{\infty}, \dots, u_{J_+}^{\infty}\}$ so that $u_j^{\infty}$ is not an eigenvector of $(\K,\a,\b)$ with positive eigenvalue. On the other hand, Lemma \ref{lem:Strong} implies $u_j^{\infty}$ is a nonzero element of $\K$ and satisfies
\[\a(u_j^{\infty},v) = \lambda_j^{\infty}\b(u_j^{\infty},v), \quad \text{for all}~ v \in \K.\]
This is a contradiction because we have shown that $\{u_1^{\infty}, \dots, u_{J_+}^{\infty},u_j^{\infty}\}$ is a linear independent set of eigenvectors with positive eigenvalues, but $(\K,\a,\b)$ only has $J_+$ positive eigenvalues. Conclude that $\lambda_j^{\infty} = \infty$.
 
\bigskip 
 
\textbf{Case 2:} Suppose that $t_j < \infty$ so that $\lambda_j^t$ exists for each $t \in (t_*,t_j)$. Suppose, towards a contradiction, that $\lambda_j^{t_{\infty}} < \infty$ and let $t_{\infty} = t_j$. By Lemma \ref{lem:Strong} part (ii), we know that $u_j^{t_\infty}$ satisfies
\[\a(u_j^{t_\infty},v) = \lambda_j^{t_\infty}\b^{t_{\infty}}(u_j^{t_\infty},v), \quad \text{for all}~ v \in \HH.\]
Thus, $(\lambda_j^{t_{\infty}},u_j^{t_{\infty}})$ is an eigenpair of $(\HH,\a,\b^{t_{\infty}})$ so that $(\HH,\a,\b^t)$ has $j$ positive eigenvalues for all $t \in (t_*,t_j]$, but at most $j-1$ for all $t > t_j$. This is a contradiction since the set $\{t \in \R : \lambda_j^t ~ \text{exists}\}$ is open by the Stability Proposition \ref{prop:StabSpec}. Conclude that $\lambda_j^{t_{\infty}} = \infty$, that is, $\lambda_j^t \nearrow +\infty$ as $t \nearrow t_j$. 

\end{proof}


\subsection*{Proof of Proposition \ref{NegConv}: Convergence of negative spectrum as \texorpdfstring{$t \to +\infty$}{Lg}}

\begin{proof}

In the fixed Hilbert space case let $\HH' \subset \HH$ be a finite dimensional subspace that intersects $\K$ trivially. In the moving Hilbert space case, assume in addition that $\HH' \subset \HH^{\infty}$. Set $J' = \dim(\HH')$ and note that $J'$ can take on any natural number up to $\codim(\K)$ and $\codim_{\HH^{\infty}}(\K \cap \HH^{\infty})$ in the fixed and moving Hilbert space cases, respectively.

The form $\C$ is positive on the unit ball of $\HH'$ by using the definition of $\K = \Ker(\c)$ and Cauchy--Schwarz. Since $\B$ and $\C$ are both norm continuous and the $\a$-unit sphere of $\HH'$ is compact, $\B$ is bounded by some number $M_{\B}$ and $\C$ attains its minimum, say $m_{\C} > 0$. Thus, $\B \leq M\C$ with $M = M_{\B}/m_{\C}$ and so $-\B^t \geq (t - M) \C$ on all of $\HH'$.

Let $t > M$ and $j \leq J'$ and suppose we are in the fixed Hilbert space case. By the variational characterization in Theorem \ref{VC} we have
\begin{align}
&\sup_{\Ss_j \subset \HH} \inf\{-\B^t(u) : u \in \Ss_j \text{ and } \A(u) = 1\} \label{eq:Sup1}\\
&\geq \sup_{\Ss_j \subset \HH'} \inf\{(t-M)\C(u): u \in \Ss_j \text{ and } \A(u) = 1\}\notag\\
&= \frac{t-M}{\lambda_j(\HH',\a,\c)}.\notag
\end{align}

Since $\C$ is positive on the unit sphere of $\HH'$ the problem $(\HH',\a,\c)$ has $J'$ positive eigenvalues. This shows that the supremum in (\ref{eq:Sup1}) is positive, and therefore, $\lambda_j(\HH,\a,-\b^t)$ exists and equals  the reciprocal of it by the variational characterization in Theorem \ref{VC}. Taking reciprocals we have \[|\lambda_{-j}(\HH,\a,\b^t)| \leq \frac{\lambda_j(\HH',\a,\c)}{t - M},\]
and sending $t \to \infty$ proves convergence to zero. Since $\lambda_{-j}(\HH,\a,\b^t) = -\lambda_j(\HH,\a,-\b^t)$, the eigenvalue monotonicity Lemma \ref{LimExists} implies $t \mapsto \lambda_{-j}^t(\HH,\a,\b^t)$ is increasing, which completes the proof in the fixed Hilbert space case. 

In the moving Hilbert space case, we can replace $\HH$ by $\HH^t$ in (\ref{eq:Sup1}) and proceed as above (up to but not including the monotonicity statement) when $t > \max\{M,T\}$, since $\HH' \subset \HH^t$. To see this recall that $\HH' \subset \HH^{\infty}$ and apply Lemma \ref{subspace} to conclude that each $u \in \HH'$ satisfies $\b(u,w) = \c(u,w) = 0$ for all $w \in \Ker(\a)$, so that $u \in \HH^t$.
\end{proof}

\subsection*{Proof of Corollary \ref{cor:NegInfLim}: Convergence of spectrum as \texorpdfstring{$t \to -\infty$}{Lg}}

\begin{proof}

If $j$ is a positive integer and $\lambda_{-j}(\HH,\a,\b^t)$ exists, observe that 
\begin{equation}
\label{eq:PosNeg}
\lambda_{-j}(\HH,\a,\b^t) = -\lambda_j(\HH,\a,-\b^t) = -\lambda_j(\HH,\a,(-\b)^{-t}),
\end{equation}
since $-\b^t = -\b - (-t)\c$. Parts (i) and (ii) of Corollary \ref{cor:NegInfLim} follow by using parts (i) and (ii) of Theorem \ref{SpecConv} to compute the right side of (\ref{eq:PosNeg}) as $t \searrow -\infty$. Using that 
\[-\lambda_j(\K,\a,-\b) = \lambda_{-j}(\K,\a,\b),\]
completes the proof when $j \leq J_+$. Similarly, part (iii) follows from using Proposition \ref{NegConv} and observing that the pair $(\b,\c)$ generates the same space $\HH^{\infty}$ as the pair $(-\b,\c)$, by Lemma \ref{subspace}.

\end{proof}


\section{\bf{Proofs of applications: Propositions \ref{prop:PDEApp} and \ref{prop:MatPen}}}
\label{sec:ProofPDEApp}

In order to apply our convergence theorems, we verify the fixed or moving Hilbert space conditions hold for the problems in Table \ref{tab:ProbTab}. The right sides of the first five partial differential equations in Table \ref{tab:ProbTab} are all generated by $\b^t(\cdot,\cdot)$, which allows us to verify many of the necessary conditions in the following lemma. The full proof of the sixth application, the Laplacian with Dynamical Boundary Conditions, will be given separately in Lemma \ref{DBCLem}. Recall that $\widetilde \K$ is the space of functions in $\K$ restricted to $\Omega$.

\begin{lemma}
\label{GenApp}
If the assumptions on $b$ and $c$ from Section \ref{AppPrec} hold and $\HH = H^k_{\partial \OO}(\OO)$ or $\HH = H^k(\OO)$ for some $k \geq 1$ then:
\begin{enumerate}
\item[(i)] $\b$ and $\c$ are well-defined and satisfy (C3);
\item[(ii)] $\K = \{u \in \HH : u \equiv 0 ~ \text{on} ~ \Omega'\}$;
\item[(iii)] if $\HH = H^k_{\partial \OO}(\OO)$ then $\widetilde \K = H^k_{\partial \Omega}(\Omega)$; if $\HH = H^k(\OO)$ then $\widetilde \K = H^k_{\Gamma}(\Omega)$;
\item[(iv)] assume that the fixed or moving Hilbert space conditions hold. If $|\{b^t > 0\}| > 0$ (or $|\{b^t < 0\}| > 0$) then $(\HH,\a,\b^t)$ and $(\HH^t,\a,\b^t)$ have infinitely many positive (or negative) eigenvalues. Similarly, if $|\{b|_{\Omega} > 0\}| > 0$ (or $|\{b|_{\Omega} < 0\}| > 0$) then $(\K,\a,\b)$ has infinitely many positive (or negative) eigenvalues;
\item[(v)] $\codim(\K)$ and $\codim_{\HH^{\infty}}(\K \cap \HH^{\infty})$ are infinite.
 
\end{enumerate}
\end{lemma}

\begin{rem}
\label{rem:Omega}
Our PDE convergence results in Proposition \ref{prop:PDEApp} likely hold under weaker assumptions on $\OO,\Omega,$ and $\Omega'$ than those stated in Section \ref{AppPrec}. For example, the bilinear forms $\b$ and $\c$ are weakly continuous even when $\OO$ is unbounded (see \cite[Lemma 2.13]{Willem} for $k = 1$ and $d \geq 3$). This suggests our results could handle problems such as the Schr\"odinger eigenvalue problem $(-\Delta + V)u^t = \lambda^t b^t u^t$ on $\OO = \R^d$ for $V$ such that $-\Delta + V$ generates a coercive and continuous bilinear form on an appropriate Sobolev space. Additionally, the condition ``$\{c > 0\}$ is open" could be eliminated by defining $\Omega' = \text{int}\big(\text{ess supp}(c) \big)$. To see this it would suffice to check parts (ii) and (iii) of the above lemma continue to hold.
\end{rem}

Recall that $\widetilde \a$ and $\widetilde \b$ are the bilinear forms $\a$ and $\b$ with the integration restricted to $\Omega$. We wish to show convergence of the spectrum of $(\HH,\a,\b^t)$ to that of $(\widetilde \K,\widetilde \a, \widetilde \b)$, but our convergence theorems give convergence to the spectrum of $(\K,\a,\b)$. This is easily overcome because the spectra of these two problems coincide. Indeed, observe that $\a(u,v) = \widetilde \a(\tilde u, \tilde v)$ and $\b(u,v) = \widetilde \b (\tilde u,\tilde v)$ for each $u,v \in \K$, where $\tilde u$ and $\tilde v$ are restrictions of $u$ and $v$ to $\Omega$ since $\Gamma$ has measure zero. Using this and the eigenequation implies that $(\K,\a,\b)$ and $(\widetilde \K,\widetilde \a, \widetilde \b)$ have the same spectrum. This shows that $(\widetilde \K,\widetilde \a, \widetilde \b)$ has a discrete spectrum of eigenvalues of finite multiplicities. The convergence theorem (Theorem \ref{SpecConv}) will imply that the eigenvalues of $(\HH,\a,\b^t)$ converge to those of the problem $(\widetilde \K,\widetilde \a, \widetilde \b)$, which proves Proposition \ref{prop:PDEApp}.

\begin{proof}[Proof of Lemma \ref{GenApp}]
(i) Observe that $\b(\cdot,\cdot)$ and $\c(\cdot,\cdot)$ are well-defined by the H\"older and Sobolev inequalities (see \cite[Theorem 8.8]{LL} and note that Lipschitz domains satisfy the interior cone condition).

(C3): Suppose $u^n \rightharpoonup u$ and $v^n \rightharpoonup v$ in $\HH$. For weak continuity of $\b$ we will show that $\int_{\OO} u^nv^n b \, dx \rightarrow \int_{\OO} uvb \, dx$. The argument is the same for $\c(\cdot,\cdot)$. By a straightforward polarization argument it is enough to show $\int (u^n)^2 b \, dx \to \int u^2 b\, dx$. Thus, to show convergence it suffices to prove that $(u^n)^2 \to u^2$ weakly in $L^q(\OO)$ when $1 < q< \infty$ and weak* in $L^{\infty}(\OO)$ when $q = \infty$, where $q$ is the H\"older conjugate exponent of $p = p(d,k)$ (defined in (\ref{eq:pdk})). 

First consider an arbitrary subsequence $\{u^{n_m}\}_m$ and note that $u^{n_m} \rightharpoonup u$ in $\HH$. It follows from the compact embedding $\HH \hookrightarrow L^2(\OO)$ that $u^{n_m} \rightarrow u$ in $L^2(\OO)$ after extracting a subsequence. Convergence in $L^2(\OO)$ guarantees the existence of a further subsequence such that $u^{n_m} \rightarrow u$ pointwise a.e. Since $\{u^{n_m}\}_m$ is bounded in $\HH$, it is also bounded in $L^q(\OO)$ by the Sobolev inequalities. It follows from the Banach--Alaoglu theorem that $L^q$-boundedness plus pointwise a.e.\ convergence implies $(u^{n_m})^2 \to u^2$ weakly in $L^q(\OO)$ when $1 < q < \infty$ and weak* in $L^{\infty}(\OO)$ when $q = \infty$ \cite[Ch. 6 Exercise 20]{Fol}.

\bigskip

(ii): The functions $u \in \K$ are the functions such that $\int_{\Omega'} c uv \, dx = 0$ for every $v \in \HH$. Let $v = u$ so that $\int_{\Omega'} c u^2 \, dx = 0$. Since $c > 0$ on $\Omega'$ we have $u = 0$ a.e.\ in $\Omega'$. Conversely, if $u = 0$ a.e.\ on $\Omega'$ then $u \in \K$.

\bigskip

(iii): See Lemma \ref{lem:TildeK}.

\bigskip

(iv): First suppose that $E = \{b^t > 0\}$ has positive measure. To show that there are infinitely many positive eigenvalues we will construct subspaces of arbitrarily large dimension on which $\B^t > 0$. 

First let $\zeta^{\epsilon}$ be the square root of an approximate identity and extend $b^t$ by zero so that $((\zeta^{\epsilon})^2 * b^t)(x) \to b^t(x)$ a.e.\ on $E$ as $\epsilon \to 0$. In particular, for each $N \in \N$ there are points $\{x_1, \dots, x_N\} \subset E$ such that $\int_{\OO} \zeta^{\epsilon}(x_i - y)^2 b^t(y) \, dy \to b^t(x_i) > 0$ for $i = 1, \dots, N$. Thus, by taking $\epsilon$ small enough  we have that $\int_{\OO} \zeta^{\epsilon}(x_i - y)^2 b^t(y) \, dy > 0$ for each $i$ and that $\zeta^{\epsilon}(x_1 - y)^2, \dots, \zeta^{\epsilon}(x_N - y)^2$ have pairwise disjoint support. Let $\zeta_i(y) = \zeta^{\epsilon}(x_i - y)$ for each $i = 1, \dots, N$ so that $\{\zeta_1, \dots, \zeta_N\}$ are linearly independent and $\B^t(\zeta_i) > 0$.

Let $N \in \N$ be arbitrary and let $\zeta_i$ for $i = 1, \dots, N$ be constructed as above. Then 
\begin{equation}
\label{eq:BtPos}
\B^t(\sum_{i = 1}^N a_i \zeta_i) = \sum_{i = 1}^N a_i^2 \B^t(\zeta_i) > 0,
\end{equation}
for every nonzero $(a_1, \dots, a_N) \in \R^N$ since the functions $\zeta_1, \dots, \zeta_N$ have disjoint support. 

In the fixed Hilbert space case, the variational characterization in Theorem \ref{VC} shows that there are at least $N$ positive eigenvalues since $\Span\{\zeta_i : i = 1, \dots, N\}$ is $N$-dimensional and $\B^t$ is uniformly positive on the $\a$-unit ball of the span. Since $N$ is arbitrary $(\HH,\a,\b^t)$ must have infinitely many positive eigenvalues. The claim for $(\K,\a,\b)$ can be proved the same way that it was proved for $(\HH,\a,\b^t)$ by working with $\b$ and $\Omega$ instead of $\b^t$ and $\OO$.

For the moving Hilbert space case, suppose that $\Ker(\a)$ has dimension $m$, that $\{w_1, \dots, w_m\}$ is a basis, and let $N \in \N$ be arbitrary. Then let $\zeta_1, \dots, \zeta_{N + m}$ be functions constructed as above. Consider the the $(N + m) \times m$ matrix with $(i,j)^{\text{th}}$-entry $\b^t(\zeta_i,w_j)$. The kernel of the transpose of this matrix has dimension at least $N$ by the rank-nullity theorem. Thus, there are linearly independent vectors $a^k = (a_1^k, \dots, a_{N+m}^k) \in \R^{N+m}$ for $k = 1, \dots, N$ such that 
\begin{equation}
\label{eq:OrthSet}
\sum_{i = 1}^{N+m} a_i^k \b^t(\zeta_i,w_j) = \b^t(\sum_{i = 1}^{N+m} a_i^k \zeta_i,w_j) = 0 \quad \text{for each} \quad j = 1, \dots, m.
\end{equation}

Let $Z_k = \sum_{i = 1}^{N+m} a_i^k \zeta_i$ so that $Z_k \in \HH^t$ by (\ref{eq:OrthSet}). The set $\{a^1, \dots, a^N\}$ is linearly independent so that $\Span\{Z_1, \dots, Z_N\}$ is an $N$-dimensional subspace of $\HH^t$. Since $\Span\{Z_1, \dots, Z_N\}$ consists of linear combinations of $\zeta_1, \dots, \zeta_{N+m}$ we know that $\B^t$ is uniformly positive on the $\a$-unit sphere of $\Span\{Z_1, \dots, Z_N\}$ by the same calculation in (\ref{eq:BtPos}). Thus, $(\HH^t,\a,\b^t)$ has at least $N$ positive eigenvalues by the variational characterization in Theorem \ref{VC}. Since $N$ is arbitrary, $(\HH^t,\a,\b^t)$ and therefore $(\HH,\a,\b^t)$ must have infinitely many positive eigenvalues. 

When $\{b^t < 0\}$ and $\{b|_{\Omega} < 0\}$ have positive measure an analogous construction can be performed to show that $(\HH,\a,\b^t)$ and $(\K,\a,\b)$ have infinitely many negative eigenvalues.

\bigskip

(v): In the fixed Hilbert space case $\codim(\K)$ is infinite because 
the smooth functions with support in $\Omega'$ are an infinite dimensional subspace of $\K^{\perp}$.

For the moving Hilbert space case, recall that 
\[\HH^{\infty} = \{u \in \HH : \b(u,w) = \c(u,w) = 0 ~ \text{for all} ~ w \in \Ker(\a)\}.\] 
To show that $\codim_{\HH^{\infty}}(\K \cap \HH^{\infty})$ is infinite, it suffices to construct arbitrarily many functions with disjoint support contained in $\Omega'$ that are both $\b$ and $\c$-orthogonal to $\Ker(\a)$. 

To see that $\codim_{\HH^{\infty}}(\K \cap \HH^{\infty}) = \infty$, we can construct matrices whose transposes have $(i,j)^{\text{th}}$-entries $\b(\zeta_i,w_j)$ and $\c(\zeta_i,w_j)$, where $1 \leq i \leq N + m$ and $1 \leq j \leq m$, similarly to before. One can show that the intersection of the kernels of these matrices has dimension of order $N$ as $N \to \infty$. Forming linear combinations of $\zeta_1, \dots, \zeta_{N+m}$ with coefficients given by vectors in the intersection produces subspaces of $\K^{\perp} \cap \HH^{\infty}$ of arbitrarily large dimension. This shows that $\codim_{\HH^{\infty}}(\K^{\perp} \cap \HH^{\infty})$ is infinite.

\end{proof}

Once we verify that the remaining fixed or moving Hilbert space conditions, either (C1) or (C1$'$), and (C2), are satisfied for each problem we will have proved Propositions \ref{prop:PDEApp} with the exception of the Dynamical Boundary Conditions problem, because Lemma \ref{GenApp} has already verified (C3) and identified the space $\widetilde \K$ associated to the limiting problem. Additionally, given the weights $b$ and $c$ the lemma determined $J_+,J_-,\codim(\K),$ the number of positive and negative eigenvalues of $(\HH,\a,\b^t)$, and $\codim_{\HH^{\infty}}(\K \cap \HH^{\infty})$ in the moving Hilbert space case, for each problem in Table \ref{tab:ProbTab}.


\begin{prop}[Schr\"odinger Operator \& Dirichlet Laplacian]
Let $b$ and $c$ be as in Section \ref{AppPrec} and $V \in L^{\infty}(\OO)$ be nonnegative. If $\HH = H^1_0(\OO)$ and $\a(u,v) = \int_{\OO} (\nabla u \cdot \nabla v + uvV) \, dx$, then the fixed Hilbert space conditions are satisfied.
\end{prop}

\begin{proof}
Continuity of $\a(\cdot,\cdot)$ follows from using boundedness of $V$. Coercivity follows using $V \geq 0$ and the Poincar\'e inequality. This shows (C1) and (C2), and (C3) was proved in Lemma \ref{GenApp}.

\end{proof}

\begin{prop}[Robin Laplacian]
Let $b$ and $c$ be as in Section \ref{AppPrec}. If $\HH = H^1(\OO)$ and $\a(u,v) = \int_{\OO} \nabla u \cdot \nabla v \, dx + \alpha \int_{\partial \OO} uv \, dS$ with $\alpha > 0$, then the fixed Hilbert space conditions are satisfied.
\end{prop}

\begin{proof}
The coercivity condition (C1) can be verified using a proof by contradiction similar to the proof of the usual Poincar\'e inequality, which can be found in \cite[\S 5.8.1]{Evans}. Alternatively, coercivity follows from a general Hilbert space coercivity theorem (see \cite[example 5]{Graser}).

Using that the trace operator is bounded from $H^1(\OO)$ into $L^2(\partial \OO)$ shows that $\a(\cdot,\cdot)$ is continuous and  verifies (C2). Recall (C3) was proved in Lemma \ref{GenApp}.

\end{proof}


\begin{prop}[Clamped Bi-Laplacian]
Let the weights $b$ and $c$ be as in Section \ref{AppPrec}. If $\HH = H^2_0(\OO)$ and $\a(u,v) = \int_{\OO} (\Delta u) (\Delta v) + \tau \nabla u \cdot \nabla v \, dx$ with $\tau \geq 0$, then the fixed Hilbert space conditions are satisfied. 
\end{prop}

\begin{proof}

Continuity of $\a(\cdot,\cdot)$ on $H^2_0(\OO)$ is immediate. For coercivity, note 
\begin{equation}
\label{BiLapId}
\int_{\OO} (\Delta u)^2 \, dx = \int_{\OO} \sum_{j,k} u_{x_j x_k}^2 \, dx,
\end{equation}
for $u \in H^2_0(\OO)$, by integration-by-parts. The coercivity condition (C1) then follows from repeated applications of the Poincar\'e inequality to the right side of (\ref{BiLapId}) and to the gradient term in $\A(u)$. Again, (C3) was proved in Lemma \ref{GenApp}.

\end{proof}


Let $\mathcal{P}_m(\OO)$ denote the space of polynomials of degree at most $m$ on $\OO$. Since $\OO$ is a bounded Lipschitz domain the following theorem will establish that the $\A$ forms are coercive on $\Ker(\a)^{\perp}$ in the moving Hilbert space applications. In this section $\perp$ denotes the $\HH$-orthogonal complement, where $\HH = H^k_{\partial \OO}(\OO)$ or $H^k(\OO)$. When $k = 1$ the result is essentially the Poincar\'e inequality on $H^1(\OO)$, for functions with mean value zero. Recall that $\OO$ is a domain and therefore is connected. 

\begin{theorem}[\protect{\cite[Corollary 1]{Graser}}]
\label{PDECoercivity}
Let $k \geq 1$. If $\a_k(u,v) = \sum_{\beta : |\beta| = k} \int_{\OO} D^{\beta}u D^{\beta}v \, dx$ then $\A_k(\cdot) = \a_k(\cdot,\cdot)$ is coercive on $\mathcal{P}_{k-1}(\OO)^{\perp}$.
\end{theorem}

\begin{prop}[Neumann Laplacian]
Let the weights $b$ and $c$ be as in Section \ref{AppPrec}. If $\HH = H^1(\OO)$ and $\a(u,v) = \int_{\OO} \nabla u \cdot \nabla v \, dx$ then the moving Hilbert space conditions are satisfied. 
\end{prop}

\begin{proof}

The continuity condition (C2) is obvious and (C3) was proved in Lemma \ref{GenApp}. The coercivity condition (C1$'$) is satisfied because $\Ker(\a)$ is 1-dimensional, consisting just of the constant functions. Since $c > 0$ on a set of positive measure, the only constant function in $\K$ is the zero function so $\Ker(\a) \cap \K$ is trivial. 

Coercivity on $\Ker(\a)^{\perp}$ follows immediately from Theorem \ref{PDECoercivity} with $k = 1$. Alternatively, it follows from the Poincar\'e inequality (see \cite[\S 5.8.1]{Evans}) and noting that the map $u \mapsto \frac{1}{|\OO|} \int_{\OO} u \, dx$ is the orthogonal projection onto the constants.
\end{proof}

\begin{prop}[Free Bi-Laplacian]
Let the weights $b$ and $c$ be as in Section \ref{AppPrec}. If $\HH = H^2(\OO)$ and $\a(u,v) = \int_{\OO} (D^2u \cdot D^2v + \tau \nabla u \cdot \nabla v) \, dx$ with $\tau \geq 0$, then the moving Hilbert space conditions are satisfied. 
\end{prop}

\begin{proof}

The continuity condition (C2) is obvious and (C3) was proved in Lemma \ref{GenApp}. To see (C1$'$), first suppose that $\tau = 0$. It is easy to see that $\Ker(\a) = \mathcal{P}_1(\OO)$. Since $c > 0$ on an open set any polynomial in $\Ker(\a) \cap \K$ must be identically zero on $\Omega'$, and so $\Ker(\a) \cap \K$ is trivial. Coercivity on $\Ker(\a)^{\perp}$ follows from noting that $\a$ is equal to $\a_2$ and applying Theorem \ref{PDECoercivity}. 

When $\tau > 0$ condition (C1$'$) follows from the $\tau = 0$ case since $\Ker(\a)$ consists of the constants (rather than all of the first degree polynomials) and the $\tau$-term only makes $\A(u)$ larger than when $\tau = 0$.
\end{proof}

\begin{prop}[Laplacian with dynamical boundary conditions]
\label{DBCLem}
Assume $\Omega = \OO$ is a bounded Lipschitz domain. If $\HH = H^1(\Omega)$, $\a(u,v) = \int_{\Omega} \nabla u \cdot \nabla v \, dx$, and $\b(u,v) = \int_{\Omega} uv \, dx $ and $\c(u,v) = \int_{\partial \Omega} uv \, dS$, then the moving Hilbert space conditions are satisfied and $\K = H^1_0(\Omega)$. When $d \geq 2$ the problem $(\HH,\a,\b^t)$ has infinitely many positive and negative eigenvalues for each $t > 0$ and $\codim_{\HH^{\infty}}(\K \cap \HH^{\infty})$ is infinite. When $d = 1$ the same holds except there is only a single negative eigenvalue for large $t$, and $\codim_{\HH^{\infty}}(\K \cap \HH^{\infty}) = 1$.
\end{prop}

\begin{proof}

Continuity of $\a(\cdot,\cdot)$ is clear. Verifying (C1$'$) is identical to the Neumann case once we show $\K = H^1_0(\Omega)$ since the constants intersect $H^1_0(\Omega)$ trivially. Indeed, the bilinear form 
\[\c(u,v) = \int_{\partial \Omega} Tu Tv \, dS,\] 
has kernel 
\[\K = \{u \in H^1(\Omega) : \int_{\partial \Omega} Tu Tv \, dS = 0 \text{ for all } v \in H^1(\Omega)\} = \{u \in H^1(\Omega) : Tu = 0\}.\]
Since $\Omega$ is Lipschitz, $H^1_0(\Omega) = \{u \in H^1(\Omega) : Tu = 0\}$ (see \cite[\S 2.4.3]{Necas}), and so 
\begin{equation}
\label{eq:DBCLimSpace}
\K = H^1_0(\Omega).
\end{equation}

To see (C3) note that weak continuity of $\b(\cdot,\cdot)$ follows from Lemma \ref{GenApp}. For weak continuity of $\c(\cdot,\cdot)$, let $\{u_n\}_n,\{v_n\}_n \subset H^1(\Omega)$ be weakly convergent sequences with limits $u$ and $v$. Since the trace map is a compact operator on Lipschitz domains \cite[\S 2.6.2]{Necas} it is also completely continuous \cite{Rud}. Therefore, $T(u_n) \to T(u)$ in $L^2(\partial \OO)$ and similarly for $v_n$. This shows that $\c(u_n,v_n) \rightarrow \c(u,v)$ and so $\c(\cdot,\cdot)$ is weakly continuous. 

The numbers of positive and negative eigenvalues follow directly from \cite[Theorem 2]{BvBR} (by setting $\sigma = -t$ for $t > |\Omega|/|\partial \Omega|$). 

To compute $\codim_{\HH^{\infty}}(\K \cap \HH^{\infty})$ recall that $H^1(\Omega) = H^1_0(\Omega) \oplus H^1_{\Delta}(\Omega)$, where 
\[H^1_{\Delta}(\Omega) = \{u \in H^1(\Omega) : \int_{\Omega} \nabla u \cdot \nabla v \, dx = 0 \text{ for each } v \in H^1_0(\Omega)\}\] 
is the subspace of harmonic functions in $H^1(\Omega)$. This also induces the decomposition $\HH^{\infty} = (H^1_0(\Omega) \cap \HH^{\infty}) \oplus (H^1_{\Delta}(\Omega) \cap \HH^{\infty})$. Together with (\ref{eq:DBCLimSpace}) we have
\[\codim_{\HH^{\infty}}(\K \cap \HH^{\infty}) = \codim_{\HH^{\infty}}(H^1_0(\Omega) \cap \HH^{\infty}) = \dim(H^1_{\Delta} \cap \HH^{\infty}).\]

By the Subspace Lemma \ref{subspace}, we know 
\[\HH^{\infty}_{\Delta} = H^1_{\Delta}(\Omega) \cap \HH^{\infty} = \{u \in H^1_{\Delta}(\Omega) : \int_{\Omega} u \, dx = 0 ~ \text{and} ~ \int_{\partial \Omega} u \, dS = 0\}.\] 
When $d \geq 2$, the space $H^1_{\Delta}(\Omega)$ is infinite dimensional as it contains the harmonic polynomials. Since $\HH^{\infty}_{\Delta}$ has codimension at most two in $H^1_{\Delta}(\Omega)$ we know $\dim(\HH^{\infty} \cap H^1_{\Delta}(\Omega)) = \infty$. When $d = 1$, the only harmonic functions are the linear polynomials so a direct calculation shows $\dim(\HH^{\infty}_{\Delta}) = 1$.

Alternatively, one could compute $\codim_{\HH^{\infty}}(H^1_0(\Omega) \cap \HH^{\infty})$ more directly. If $d = 1$ and $\Omega = (-1,1)$ one can show that an element of $\HH^{\infty}$ differs from an element of $H^1_0(\Omega) \cap \HH^{\infty}$ by a linear function, which shows that $\codim_{\HH^{\infty}}(H^1_0(\Omega) \cap \HH^{\infty}) = 1$. If $d \geq 2$ then one can construct infinitely many functions in $\HH^{\infty}$ that have disjoint supports and are nonzero on the boundary. These functions can not differ from each other by elements of $H^1_0(\Omega)$. Therefore, their span is an infinite dimensional subspace of the quotient so that $\codim_{\HH^{\infty}}(H^1_0(\Omega) \cap \HH^{\infty}) = \infty$.
\end{proof}

\subsection*{Proof of Proposition \ref{prop:MatPen}: Matrix Pencil Convergence}

Part (ii) of the proposition follows immediately from the Convergence Theorem \ref{SpecConv}. To show the remaining parts let $J_+^t$ denote the number of positive eigenvalues of $(\R^d,\a,\b^t)$ and similarly for $J_-^t$ and $J_{\infty}^t$.

First we will show that $J^t_{\infty} = \dim(\Ker(\b^t)) = 0$ for all large enough $t$. By definition of $B^t$
\[\Ker(\b^t) = \{v \in \R^d : C v = t^{-1} B v\}.\]
Hence, it suffices to show that the eigenvalue problem $(\R^d,\c,\b)$ only has finitely many nonzero finite eigenvalues.

Let $v$ be a (nonzero) eigenvector of $(\R^d,\c,\b)$ with eigenvalue $\mu$. If $\mu \neq 0$ we must have $C v \neq 0$ since $\Ker(B) \cap \Ker(C)$ is trivial so that $\C(v) > 0$ by Cauchy--Schwarz for $\c(\cdot,\cdot)$. Note that eigenvectors of $(\R^d,\c,\b)$ with distinct eigenvalues are $\c$-orthogonal. It follows that eigenvectors with distinct nonzero eigenvalues must be linearly independent. Thus, there can only be finitely many nonzero eigenvalues of $(\R^d,\c,\b)$ since $\R^d$ is finite dimensional.

Next we will show that there is a $T$ such that 
\begin{equation}
\label{eq:JtBound}
J_+^t \leq J_+ + J_{\infty}, \quad \text{for all} ~ t > T.
\end{equation}
Recall that each $\lambda_j^t$ with $j > J_+$ tends to $+\infty$ as $t \nearrow t_j$ for some $t_j \in (-\infty,+\infty]$ by Convergence Theorem \ref{SpecConv}. Since $J_+^t$ is a decreasing function of $t$ by Proposition \ref{prop:StabSpec} there can only be finitely many eigenvalues that tend to $+\infty$ in finite time. Thus, there is a $T$ such for all $t > T$ each positive eigenvalue of $(\R^d,\a,\b^t)$ either: converges to a positive eigenvalue of $(\Ker(\c),\a,\b)$ or tends to $+\infty$ as $t \to \infty$. By Lemma \ref{lem:Strong}, the positive eigenvalues that tend to $+\infty$ in infinite time have a subsequence of eigenvectors that converge strongly to an element of $\Ker(\b|_{\K \times \K})$, since weak convergence is equivalent to strong convergence in finite dimensions. There can be at most $J_{\infty} = \dim(\Ker(\b|_{\K \times \K}))$ such eigenvalues because the approximating eigenvectors can be chosen to be $\a$-orthonormal so that they converge to a linearly independent subset of $\Ker(\b|_{\K \times \K})$. This shows $J_+^t - J_+ \leq J_{\infty}$ and proves (\ref{eq:JtBound}).

Since $J_+^t + J_-^t + J_{\infty}^t = d = J_+ + J_{\infty} + J_- + \rank(C)$, inequality (\ref{eq:JtBound}) shows that
\begin{equation}
\label{eq:JtBound2}
J_-^t + J_{\infty}^t \geq d - (J_+ + J_{\infty}) = \rank(C) + J_-, \quad \text{for all}~ t > T.
\end{equation}

Since $J_{\infty}^t = 0$ for large $t$ we can increase $T$ if necessary so that (\ref{eq:JtBound2}) becomes
    \[J_-^t \geq \rank(C) + J_-, \quad \text{for all}~ t > T.\]
    
By Proposition \ref{NegConv} we know at least $\rank(C)$ negative eigenvalues increase to zero. On the other hand, at most $\rank(C)$ negative eigenvalues increase to zero since the corresponding eigenvectors can be chosen to be $\a$-orthogonal, and therefore a subsequence converge strongly to a linearly independent subset of $\Ker(\c)^{\perp_{\a}}$ by Lemma \ref{lem:ConNEv}. Thus, exactly $\rank(C)$ negative eigenvalues increase to zero. This leaves at least $J_-$ eigenvalues that do not tend to zero. By Lemma \ref{lem:ConNEv} each of these eigenvalues tend to a negative eigenvalue of the limiting problem. By extracting strongly convergent subsequences of the eigenvectors, there can be at most $J_-$ by pairing each approximating eigenvalue with the negative eigenvalue it converges to. 

This shows that the eigenvalues $\lambda^t_{-(\rank(C) + i)} \nearrow \lambda_{-i}$ for each $i = 1, \dots, J_-$. This leaves exactly $J_{\infty}$ positive eigenvalues remaining that must increase to $+\infty$ as $t \to +\infty$. \qed

\remark[Finite-time blow-up]{Suppose that $\lim_{t \nearrow t_j} \lambda_j^t = + \infty$ for some $t_j < \infty$. It follows from the Stability Proposition \ref{prop:StabSpec} that $\lambda_j^t$ exists on $(-\infty,t_j)$ but not for $t \geq t_j$. Since $\Ker(\b^t)$ is nontrivial for only finitely many $t$ by the proof of Proposition \ref{prop:MatPen} above, there is a neighborhood of $t_j$ such that $\Ker(\b^t)$ is nontrivial only at $t_j$ on that neighborhood. Thus, in order for $(\R^d,\a,\b^t)$ to have $d$ eigenvalues for $t > t_j$ there must be a negative eigenvalue that exists on $(t_j,+\infty)$ but not for $t \leq t_j$. By Corollary \ref{cor:NegInfLim}, this negative eigenvalue must tend to $-\infty$ as $t \searrow t_j$. This explains the ``reappearing" phenomenon mentioned in the caption of Figure \ref{fig:MatPen}.}


\section*{Acknowledgments}
The author would like to thank Richard Laugesen for his guidance on the writing of this paper and Giles Auchmuty for helpful email correspondences. Support from the U.S. Department of Education through the Graduate Assistance in Areas of National Need (GAANN) program and the University of Illinois Campus Research Board award RB19045 (to Richard Laugesen) is gratefully acknowledged.


\appendix
\section{Identification of \texorpdfstring{$\widetilde \K$: the restriction of $\K$ to $\Omega$}{Lg}}


The next lemma shows that the trace from one side of $\Gamma$ equal the trace from the other side. Recall that $\OO$ is a bounded Lipschitz domain, $\Omega'$ and $\Omega = \OO \setminus \overline{\Omega'}$ are open sets with Lipschtiz boundaries, and $\Gamma = \partial \Omega \cap \partial \Omega'$. 

\begin{lemma}
\label{lem:TrLem}
Let $T : H^1(\Omega) \rightarrow L^2(\partial \Omega)$ and $T' : H^1(\Omega') \rightarrow L^2(\partial \Omega')$ denote the trace operators and let $k \geq 1$. If $v \in H^k(\OO)$ then
\[T(D^{\beta}v|_{\Omega}) \big|_{\Gamma} = T'(D^{\beta}v|_{\Omega'}) \big|_{\Gamma},\]
for each multiindex $\beta$ with $|\beta| \leq k-1$. Additionally, $\partial \Omega = \Gamma \cup (\partial \Omega \cap \partial \OO)$.
\end{lemma}

\begin{proof}
Since $C^{\infty}(\OO)$ is dense in $H^k(\OO)$ there is a sequence $\{\varphi_n\}_n \subset C^{\infty}(\OO)$ such that $\varphi_n \rightarrow v$ in the $H^k$-norm. In particular, 
\[\lVert (v - \varphi_n) \big|_{\Omega} \rVert_{H^k} \to 0 \quad \text{and} \quad \lVert (v - \varphi_n) \big|_{\Omega'}\rVert _{H^k} \to 0.\]
Since the trace is a bounded operator on $H^1$ for bounded open sets with Lipschitz boundaries (see \cite[\S 4.3]{EG}) this shows that 
\begin{equation}
\label{TrLim}
\lVert T(D^{\beta}(v - \varphi_n)|_{\Omega})  \rVert _{L^2(\Gamma)} \to 0 \quad \text{and} \quad \lVert T'(D^{\beta}(v - \varphi_n)|_{\Omega'})  \rVert _{L^2(\Gamma)} \rightarrow 0.
\end{equation}

Observe that because $D^{\beta} \varphi_n$ is continuous, 
\[T(D^{\beta}(v - \varphi_n)|_{\Omega})   = T(D^{\beta}v|_{\Omega})   - D^{\beta}\varphi_n   \quad \text{and} \quad T'(D^{\beta}(v - \varphi_n)|_{\Omega'})   = T'(D^{\beta}v|_{\Omega'})   - D^{\beta}\varphi_n  . \]
Using these equalities and (\ref{TrLim}) shows that
\[T(D^{\beta}v|_{\Omega}) \big|_{\Gamma} = \lim_{n \rightarrow \infty} D^{\beta}\varphi_n \big|_{\Gamma} =  T'(D^{\beta}v|_{\Omega'}) \big|_{\Gamma} \quad \text{in}~ L^2(\Gamma).\]

To see that $\partial \Omega = \Gamma \cup (\partial \Omega \cap \partial \OO)$ observe that the inclusion ``$\supset$'' follows immediately from the definition of $\Gamma$. To prove the forward inclusion let $x \in \partial \Omega$. If $x \in \partial \OO$ we are done so suppose that $x \in \OO$. Since $\OO = \Omega \sqcup (\overline{\Omega'} \cap \OO)$ every small enough ball centered at $x$ intersects both $\Omega$ and $\overline{\Omega'} \cap \OO$ nontrivially. None of these balls are contained in $\overline{\Omega'}$ because each one intersects $\Omega$. Thus, $x \in \partial \Omega'$, and so $x \in \Gamma$.

\end{proof}

Recall that $\K$ is the subspace of $\HH = H^k(\OO)$ or $H^k_{\partial \OO}(\OO)$ consistsing of functions that vanish on $\Omega'$, and $\widetilde \K$ is the space of functions in $\K$ restricted to $\Omega$.

\begin{lemma}
\label{lem:TildeK}
Under the above conditions on $\OO,\Omega,$ and $\Omega'$ we have that
\[\widetilde \K = H^k_{\Gamma}(\Omega) ~ \text{or} ~ H^k_{\partial \Omega}(\Omega), \quad \text{for} ~ k \geq 1,\]
when $\HH = H^k(\OO)$ or $H^k_{\partial \OO}(\OO)$, respectively.
\end{lemma}

\begin{proof}
First we show that $\widetilde \K \subset H^k_{\Gamma}(\Omega)$. Let $\tilde v \in \widetilde \K$ be the restriction of $v \in \K$ to $\Omega$. Note $v \equiv 0$ on $\Omega'$. By Lemma \ref{lem:TrLem}, $T((D^{\beta}v)|_{\Omega})$ is zero a.e.\ on $\Gamma$ for each multi-index $\beta$ with $|\beta| \leq k-1$. This shows that $\tilde v = v|_{\Omega} \in H^k_{\Gamma}(\Omega)$. In particular, this shows that if $v \in H^k_{\partial \OO}(\OO)$ then $\tilde v = v|_{\Omega} \in H^k_{\partial \Omega}(\Omega)$ since $\partial \Omega  = \Gamma \cup (\partial \Omega \cap \partial \OO)$ by Lemma \ref{lem:TrLem}.

To show that $H^k_{\Gamma}(\Omega) \subset \widetilde \K$, we will show that if $u \in H^k_{\Gamma}(\Omega)$ then its extension by zero to all of $\OO$ is an element of $H^k(\OO)$. Since $\lVert u\rVert _{H^k(\Omega)} < \infty$ already, it is enough to show that 
\[u^E = \begin{cases}
      u \quad \text{on} \quad \Omega\\
      0 \quad \text{on} \quad \OO \setminus \Omega,
\end{cases}
\]
has weak derivatives of order $k$. The remainder of the proof will proceed by induction on $k$.

\bigskip

\textbf{Base Case:} First assume that $k = 1$. Let $\partial_i u$ denote the weak partial derivative of $u$ in the variable $x_i$. We will show that 
\[(\partial_i u)^E = \begin{cases}
       \partial_iu \quad \text{on} \quad \Omega\\
      0 \quad \text{on} \quad \OO \setminus \Omega,
\end{cases}
\]
is the weak derivative of $u^E$. 

Since $\Omega$ has Lipschitz boundary, the usual integration-by-parts formula holds for every $\varphi \in C^{\infty}_c(\OO)$ (see \cite[\S 4.3]{EG}) so that
\begin{equation}
\label{eq:IBP}
\int_{\OO} u^E \partial_i \varphi \, dx = \int_{\Omega} u \partial_i \varphi \, dx = -\int_{\Omega} (\partial_i u) \varphi  \, dx + \int_{\partial \Omega} (Tu) \hspace{.5mm} \varphi \hspace{.5mm} n_i \, dS,
\end{equation}
where $n_i$ is the $i^{\text{th}}$-component of the outward unit normal vector to $\partial \Omega$.

Observe that $Tu|_{\Gamma} = 0$, and $\varphi|_{\partial \Omega \cap \partial \OO} = 0$ since $\varphi$ has compact support in $\OO$. Since $\partial \Omega \subset \Gamma \cup (\partial \Omega \cap \partial \OO)$ by Lemma \ref{lem:TrLem} the integrand of the boundary term in (\ref{eq:IBP}) is zero. Thus, 
\[\int_{\OO} u^E \partial_i \varphi \, dx = -\int_{\Omega} \partial_i u \, \varphi \, dx =  -\int_{\OO} (\partial_i u)^E \varphi \, dx, \quad \text{for every~} \varphi \in C_c^{\infty}(\OO),\]
which shows that the weak derivative of $u^E$ exists and equals $(\partial_i u)^E$, so that $u^E \in H^1(\OO)$.

\bigskip

\textbf{Inductive Step:} Suppose that $k > 1$ and let $u \in H^k_{\Gamma}(\Omega)$. In particular, $u \in H^{k-1}_{\Gamma}(\Omega)$, and so $u^E \in H^{k-1}(\OO)$ and has weak derivatives $D^{\beta} u^E = (D^{\beta} u)^E$ for each multi-index with $|\beta| \leq k-1$ by the inductive hypothesis. Since $D^{\beta} u \in H^1_{\Gamma}(\Omega)$, the base case implies that $(D^{\beta}u)^E \in H^1(\OO)$ so $D^{\beta}u^E \in H^1(\OO)$. Thus, $u^E \in H^k(\OO)$. 

\end{proof}

\section{Stability of spectrum}

The next proposition proves stability results for the spectrum in both the fixed and moving cases.

\begin{prop}[Stability of spectrum]
\leavevmode
\label{prop:StabSpec}

\noindent(i) Assume the fixed Hilbert space conditions hold. Then $(\HH,\a,\b^t)$ has at least $J_+$ positive and $J_-$ negative eigenvalues for all $t \in \R$. If $\lambda_j^{t_*}$ exists for some $t_*$ then $\lambda_j^t$ exists on $(-\infty,t_* + \delta)$ for some $\delta > 0$. Similarly, if $\lambda_{-j}^{t_*}$ exists for some $t_*$ then $\lambda_{-j}^t$ exists on $(t_* - \delta,+\infty)$ for some $\delta > 0$. Additionally, the number of positive and negative eigenvalues are decreasing and increasing functions of $t$, respectively.\\
(ii) Assume the moving Hilbert space conditions hold. Then $(\HH,\a,\b^t)$ has at least $J_+$ positive and $J_-$ negative eigenvalues for all sufficiently large positive and negative $t$, respectively. If $\lambda_j^{t_*}$ exists for some sufficiently large $t_* > 0$ then $\lambda_j^t$ exists on some open interval around $t_*$. Similarly, if $\lambda_{-j}^{t_*}$ exists for some sufficiently negative $t_* < 0$ then $\lambda_{-j}^t$ exists on some open interval around $t_*$. Additionally, the number of positive and negative eigenvalues are decreasing and increasing functions of $t$ for all sufficiently large positive and negative $t$, respectively. 
\end{prop}

\begin{proof}

We first prove Proposition \ref{prop:StabSpec} for the positive eigenvalues. Recall that $J_+$ is the number of positive eigenvalues of $(\K,\a,\b)$, and let $J = J_+$ until the end of the proof for notational ease. There is nothing to prove when $J = 0$ so assume that $J \geq 1$. 

Let $t$ be fixed ($t > T$ in the moving Hilbert space case). By applying the variational characterization in Theorem \ref{VC} to $(\K,\a,\b)$ we know that $\B^t = \B$ is positive on the $\a$-unit sphere of some $J$-dimensional subspace $\Ss_J \subset \K = \Ker(\c)$.

The same variational characterization applied to $(\HH,\a,\b^t)$ shows that $(\HH,\a,\b^t)$ has at least $J$ positive eigenvalues. The same holds in the Moving Hilbert space case by the variational characterization in Theorem \ref{MVC} since $\K \cap \Ker(\a)$ being trivial implies that $\Ss_J \cap \Ker(\a)$ is trivial.

\bigskip

Suppose that $\lambda_j^{t_*}$ exists. To see that $\lambda_j^t$ exists in an open interval around $t_*$, observe that by the variational characterizations in Theorems \ref{VC} and \ref{MVC} there is a $j$-dimensional subspace $\Ss_j$ so that $\delta = \inf\{\B^{t_*}(u) : u \in \Ss_j, \A(u) = 1\} > 0$, with $\Ss_j \cap \Ker(\a) = \{0\}$ in the moving Hilbert space case. Let $M_{\C} = \max\{\C(u) : u \in \Ss_j, \A(u) = 1\}$ so that $\B^t > 0$ on the
$\a$-unit sphere of $\Ss_j$ when $t \in (-\infty, t_* + \frac{\delta}{M_{\C}})$. The variational characterizations show that $(\HH,\a,\b^t)$ has $j$ eigenvalues for each $t \in (-\infty, t_* + \frac{\delta}{M_{\C}})$ in the fixed Hilbert space case and $(\HH^t,\a,\b^t)$ has $j$ eigenvalues for each $t \in (T,t_* + \frac{\delta}{M_{\C}})$ in the moving Hilbert space case. 

\bigskip
Let $J_+^t$ denote the number of positive eigenvalues of $(\HH,\a,\b^t)$. To see that $J^t_+$ is a decreasing function of $t$ suppose, towards a contradiction, that $s_2 > s_1$ (both larger than $T$ in the moving Hilbert space case) are such that $J_+^{s_2} > J_+^{s_1}$. Thus, there is an eigenvalue that exists at $t = s_2$ but not at $t = s_1$, which contradicts the above stability result. Conclude that $J_+^t$ is decreasing in $t$.

\bigskip

To see the analogous statements for the negative eigenvalues note that 
\begin{equation}
\label{eq:PosNegId}
\lambda_j(\HH,\a,(-\b)^{-t}) = -\lambda_{-j}(\HH,\a,-\b^t).
\end{equation}
In the fixed Hilbert space case, applying the above result for the positive eigenvalues to $\lambda_j(\HH,\a,(-\b)^{-t})$ we see that it exists for each $j = 1, \dots, J_+(\K,\a,-\b)$. Since $J_-(\K,\a,\b) = J_+(\K,\a,-\b)$, using $(\ref{eq:PosNegId})$ we obtain the result for the negative eigenvalues. Existence of $\lambda^t_{-j}$ for $t \in (t_* - \delta,+\infty)$ for some $\delta > 0$ follows from the result for positive eigenvalues and (\ref{eq:PosNegId}). Since the number of negative eigenvalues of $(\HH,\a,\b^t)$ is equal to the number of positive eigenvalues of $(\HH,\a,(-\b)^{-t})$ by (\ref{eq:PosNegId}) we have the monotonicity result for the number of negative eigenvalues.

The same holds for $(\HH^t,\a,\b^t)$ in the moving Hilbert space case, but now we must take $t < -T$ to apply the above result to $(\HH^t,\a,-\b-(-t)\c)$.

\bigskip

Each eigenvalue has finite multiplicity due to \cite[Theorem 4.3]{Auch}.

\end{proof}

\section{Lipschitz continuity of spectrum}

\begin{prop}[Lipschitz continuity of spectrum]
\label{ContSpec}
Assume that $j \geq 1$ and that $\lambda_j^{t_*}$ exists for some $t_*$. If the fixed Hilbert space conditions hold then the functions $t \mapsto 1/\lambda^t_{\pm j}$ are Lipschitz continuous with constant $1/\lambda_1(\HH,\a,\c)$ whenever they exist. If the moving Hilbert space conditions hold and $t_* > 0$ is sufficiently large then $t \mapsto 1/\lambda^t_{\pm j}$ is Lipschitz continuous on a neighborhood of $\pm t_*$.
\end{prop}
\noindent In the moving Hilbert space case it is possible that the functions $t \mapsto 1/\lambda^t_{\pm j}$ are only locally Lipschitz continuous when they exist since there can be curves of eigenvalues passing through zero. See point 1 in the discussion at the end of Section \ref{sec:MR}.

\begin{proof}
First we prove the proposition for the positive eigenvalues. By Proposition \ref{prop:StabSpec} there exists a $\delta > 0$ such that $\lambda_j^t$ exists on $I_F = (-\infty,t_* + \delta)$ in the fixed Hilbert space case, and on $I_M = (t_* - \delta, t_* + \delta) \subset (T,+\infty)$ in the moving Hilbert space case since we can assume $t_* > T$. Let $s,t \in I_F$ or $I_M$ and set $\widehat{\Ss_j} = \Span\{u_1^s, \dots, u_j^s\}$ where $\{u_1^s, \dots, u_j^s\}$ is a collection of $j$ linearly independent eigenvectors corresponding to the first $j$ positive eigenvalues of $(\HH,\a,\b^s)$. Observe that $\B^t(u) = -(t - s)\C(u) + \B^s(u)$ by adding and subtracting $s \C(u)$. 

By the variational characterizations in Theorems \ref{VC} and \ref{MVC} and noting that $\widehat{\Ss_j}$ intersects $\Ker(\a)$ trivially since $\widehat{\Ss_j} \subset \HH^s$ we have
\begin{align}
&\frac{1}{\lambda_j^t} = \sup_{\Ss_j : \Ss_j \cap \Ker(\a) = \{0\}} \inf \{\B^t(u) : u \in \Ss_j \text{ and } \A(u) =1\}\notag\\
&\geq \inf \{\B^t(u) : u \in \widehat{\Ss_j} \text{ and } \A(u) =1\}\notag\\
&= \inf \{-(t - s)\C(u) + \B^s(u) : u \in \widehat{\Ss_j} \text{ and } \A(u) =1\}\notag\\
&\geq -|t-s|\sup\{\C(u) : u \in \widehat{\Ss_j} \text{ and } \A(u) =1\} + \inf\{ \B^s(u) : u \in \widehat{\Ss_j} \text{ and } \A(u) =1\}\notag\\
&= -|t-s|\sup\{\C(u) : u \in \widehat{\Ss_j} \text{ and } \A(u) =1\} + \frac{1}{\lambda_j^s}.\notag
\end{align}

Expanding the supremum from the $\a$-unit sphere of $\widehat \Ss_j$ to that of $\HH$ or $\HH^s$ shows that 
\begin{equation}
\label{eq:Lip}
\frac{1}{\lambda_j^t} \geq - \mu |t-s| + \frac{1}{\lambda_j^s}, \quad \text{where} ~ \mu = 1/\lambda_1(\HH,\a,\c) ~ \text{or} ~ \sup_{s \in I_M} \{1/\lambda_1(\HH^s,\a,\c)\},
\end{equation}
in the fixed or moving cases, respectively. To see this note that the problems $(\HH,\a,\c)$ and $(\HH^s,\a,\c)$ for $s \in I_M$ satisfy the fixed and moving Hilbert space conditions, and each have at least one positive eigenvalue given by the expanded supremum since $\C(\cdot)$ is positive somewhere on $\HH$ and $\HH^s$. The supremum over $s \in I_M$ is finite since $\c(\cdot,\cdot)$ is a bounded bilinear form by weak continuity, so that $\C(u) \lesssim \lVert u \rVert^2 \lesssim \A(u) \lesssim 1$ since $\A$ is uniformly coercive on $\HH^s$ by Lemma \ref{MCoercivity}.

At this point we have made no assumptions about the relation between $t$ and $s$. Therefore, we can combine (\ref{eq:Lip}) and the same inequality with $s$ and $t$ swapped so that
\[\left| \frac{1}{\lambda_j^s} - \frac{1}{\lambda_j^t} \right| \leq \mu |t-s|. \]
This shows that $t \mapsto 1/\lambda_j^t$ is Lipschitz continuous as stated in the proposition. Moreover, in the fixed case we can take $I_F$ to be the maximal set on which $\lambda_j^t$ exists to see $t \mapsto 1/\lambda_j^t$ is uniformly Lipschitz with constant $1/\lambda_1(\HH,\a,\c)$. 

The statements for the negative eigenvalues follow by applying the above continuity results for positive eigenvalues to the right side of
\[\lambda_{-j}(\HH,\a,\b^t) = -\lambda_j(\HH,\a,-(\b-t\c)) = - \lambda_j(\HH,\a,(-\b)^{-t}).\]
\end{proof}


\end{document}